\documentclass[reqno,11pt]{amsart}

\usepackage{tikz}
\usepackage{pgfplots}
\usepackage{ulem}
\usepackage{amsmath}
\usepackage{amssymb}
\usepackage{mathtools}
\usepackage{color}
\usetikzlibrary{arrows,
	arrows.meta,
	bending,calc,patterns,angles,quotes}

\usepackage{amsmath, amssymb, amsthm, amsfonts} 
\usepackage[english]{babel}
\usepackage{bbm}
\usetikzlibrary{arrows, automata,positioning,calc,shapes,decorations.pathreplacing,decorations.markings,shapes.misc,petri,topaths}
  \usepackage{pgfplots}
  \pgfplotsset{compat=newest}
  \usetikzlibrary{plotmarks}
  \usepackage{grffile}
\newlength\figureheight
  \newlength\figurewidth
\pgfplotsset{%
    tick label style={font=\scriptsize},
    label style={font=\footnotesize},
    legend style={font=\footnotesize},
         every axis plot/.append style={very thick}
}

\usepackage{rotating}
\usepackage{amsbsy,enumerate}
\usepackage{graphicx}
\usepackage{ccaption}
\usepackage{comment}
\usepackage{mathrsfs} 
\usepackage{diagbox}
\usepackage{caption}
\usepackage{thmtools}
\usepackage{amsthm}
\usepackage{graphicx}
\usepackage{url}
\usepackage{epstopdf}
\usepackage[a4paper,bindingoffset=0.5cm,left=2cm,right=2cm,top=2.5cm,bottom=2cm,footskip=.8cm]{geometry}

\usepackage{tikz}
\usetikzlibrary{arrows, automata,positioning,calc,shapes,decorations.pathreplacing,decorations.markings,shapes.misc,petri,topaths}
  \usepackage{pgfplots}
  \pgfplotsset{compat=newest}
  \usetikzlibrary{plotmarks}
  \usepackage{grffile}
\pgfplotsset{%
    tick label style={font=\scriptsize},
    label style={font=\footnotesize},
    legend style={font=\footnotesize},
         every axis plot/.append style={very thick}
}

\usepackage{rotating}
\usepackage{amsbsy,enumerate}
\usepackage{graphicx}
\usepackage{ccaption}
\usepackage{comment}
\usepackage{mathrsfs} 
\usepackage{diagbox}
\usepackage{caption}
\usepackage{algorithmic}
\usepackage[]{algorithm2e}
\usepackage{thmtools}
\usepackage{amsthm}

\newcommand{\vb}{\vspace{3.2mm}}

\renewcommand\thmcontinues[1]{Continued}

\allowdisplaybreaks

\newtheorem{lemma}{Lemma}

\newtheorem{corollary}{Corollary}

\newtheorem{theorem}{Theorem}
\newtheorem{remark}{Remark}

\makeatletter
\renewcommand{\fnum@figure}[1]{\textbf{\figurename~\thefigure}. }
\renewcommand{\fnum@table}[1]{\textbf{\tablename~\thetable}. }
\makeatother

\begin{document}

\title[Moments of polynomial functionals of L\'evy processes]{Moments of polynomial functionals \\ of spectrally positive L\'evy processes}

\author[{P. Glynn, R. Jacobovic and M. Mandjes}]{Peter Glynn, Royi Jacobovic and Michel Mandjes}

\begin{abstract} 
Let $J(\cdot)$ be a compound Poisson process with rate $\lambda>0$ and a jumps distribution $G(\cdot)$ concentrated on $(0,\infty)$. In addition, let $V$ be a random variable which is distributed according to $G(\cdot)$ and independent from $J(\cdot)$. Define a new process $W(t)\equiv W_V(t)\equiv V+J(t)-t$, $t\geqslant 0$ and let $\tau_V$ be the first time that $W(\cdot)$ hits the origin. A long-standing open problem due to Iglehart (1971) and Cohen (1979) is to derive the moments of the functional $\int_0^\tau W(t)\,{\rm d}t$ in terms of the moments of $G(\cdot)$ and $\lambda$. In the current work, we solve this problem in much greater generality, i.e., first by letting $J(\cdot)$ belong to a wide class of  spectrally positive \color{black}  L\'evy processes and secondly, by considering more general class of functionals. We also supply several applications of the existing results, e.g., in studying the process $x\mapsto \int_0^{\tau_x}W_x(t)\,{\rm d}t$ defined on $x\in[0,\infty)$.
\vb

\vb

\noindent
{\sc \textbf{AMS Subject Classification (MSC2010):}} 60J65, 60J75, 60K25, 91B30.
\vb

\noindent
{\sc Affiliations.} Peter Glynn is with the Department of Management Science and Engineering and Institute for Computational and Mathematical Engineering, 
Stanford University, United States of America. 

\noindent
Royi Jacobovic is with School of Mathematical Sciences, Tel Aviv University, P.O. Box 39040, Tel Aviv 6997801, Israel. 

\noindent Michel Mandjes is with the Mathematical Institute, Leiden University, P.O. Box 9512,
2300 RA Leiden,
The Netherlands, and with the Korteweg-de Vries Institute for Mathematics, University of Amsterdam, Science Park 904, 1098 XH Amsterdam, the Netherlands. He is also affiliated to E{\sc urandom}, Eindhoven University of Technology, Eindhoven, the Netherlands, and the Amsterdam Business School, Faculty of Economics and Business, University of Amsterdam, Amsterdam, the Netherlands.

\noindent
The research of RJ and MM was supported by the European Union’s Horizon 2020 research and innovation programme under the Marie Skłodowska-Curie grant agreement no. 945045, and by the NWO Gravitation project NETWORKS under grant no. 024.002.003.  
 \includegraphics[height=1em]{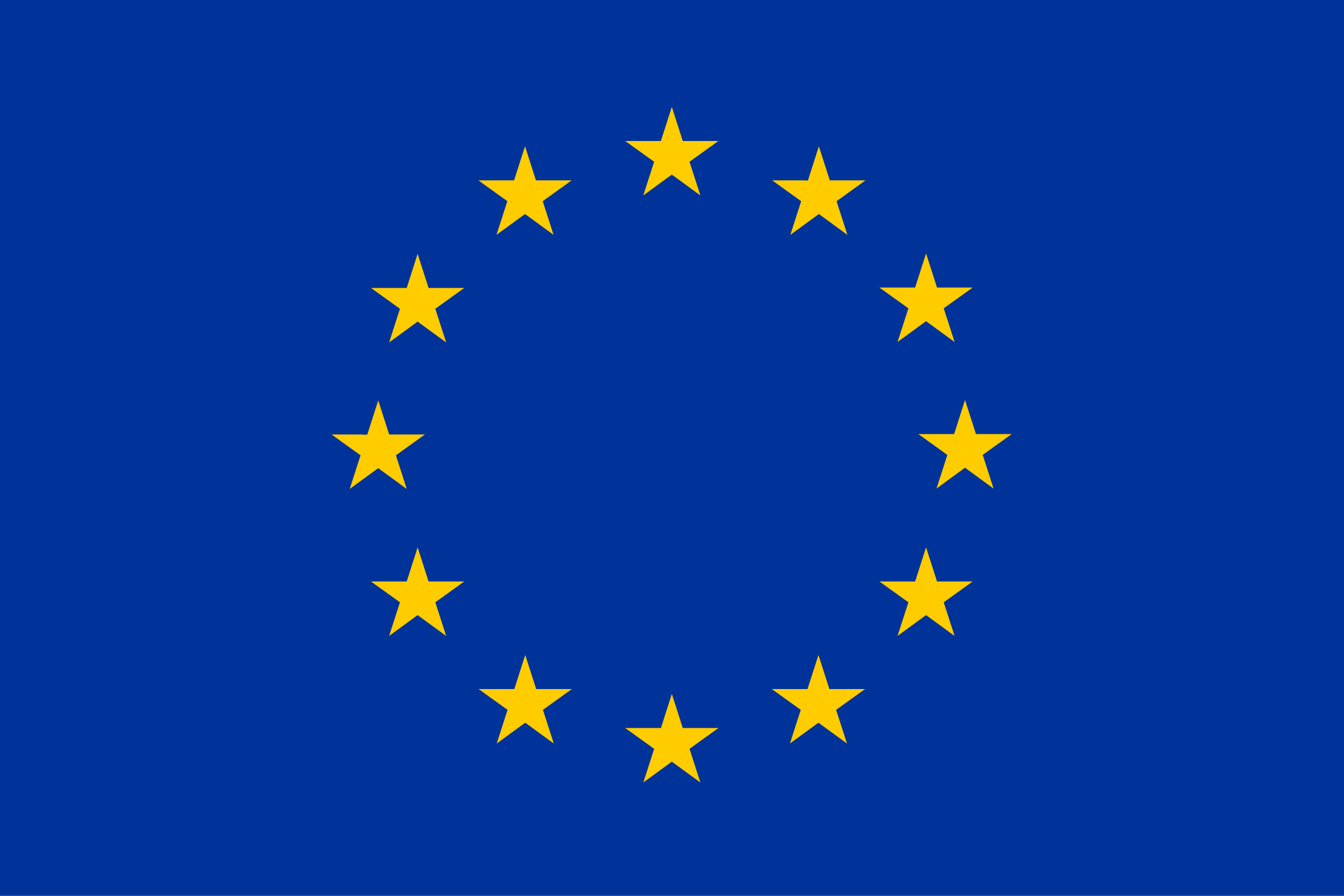} 
 
 \noindent RJ also acknowledges the support of the Israel Science Foundation, Grant \#3739/24. 

\noindent
Email addresses:  \url{glynn@stanford.edu, royijacobo@tauex.tau.ac.il, m.r.h.mandjes@uva.nl}.

\vb

\noindent
{\it Date}: {\today}.
\end{abstract}
\maketitle 

\newpage

\section{Introduction}\label{sec:introduction} 
L\'evy processes form a class of stochastic processes that start from zero, that have c\`adl\`ag sample-paths, and that have stationary independent increments; see e.g.\ the general surveys \cite{Applebaum2009,Bertoin1996}. They play a prominent role across various subdisciplines of applied probability, such as ruin theory \cite{Asmussen2010,Kyprianou2006} and queueing theory \cite{Debicki2015}, but they are also extensively used in mathematical finance \cite{Schoutens2003}. 

The current work aims to analyze the area under the graph of the L\'evy process until a stopping time. More specifically, in the setting considered we give the process an initial push $V>0$ at time~$0$, and record the L\'evy process (assumed to have a negative drift) until the first time $\tau$ that the level $0$ is hit. The main object of study is the joint distribution of the area under the graph of the L\'evy process in the interval between $0$ and $\tau$, the number of jumps of the L\'evy process in the same interval, and the hitting time $\tau$ itself. The objective is to characterize this joint distribution through a constructive scheme via which a family of joint moments can be recursively identified.

Our model is formally introduced as follows. We let $J(\cdot)$ denote a L\'evy process which is a superposition of a subordinator (i.e., a nondecreasing L\'evy process) and a drifted Brownian motion (BM). For this specific subclass of spectrally positive L\'evy processes it is well-known (see, e.g., \cite[Theorem 2.1, Lemma 2.14]{Kyprianou2006}) that the Laplace exponent  of $J(\cdot)$ necessarily takes the form 
\begin{equation}\label{eq: varphi}
    \varphi(\alpha)\equiv\varphi_{c,\sigma,\nu}(\alpha)\equiv-\log \mathbb{E}e^{-\alpha J(1)}=c\alpha-\frac{\sigma^2\alpha^2}{2}-\int_{(0,\infty)}\left(e^{-\alpha y}-1\right)\nu({\rm d}y)\ \ , \ \ \alpha\geqslant0\,,
\end{equation}
where $c\in\mathbb{R}$, $\sigma\geqslant0$ and  $\nu$ is a Borel measure on $(0,\infty)$ \color {black}  
 such that
\begin{equation}
\int_{(0,\infty)}\left(1\wedge y\right)\nu({\rm d}y)<\infty\,.   
\end{equation}
This means that $J(\cdot)$ could for instance be a compound Poisson process (CPP) with nonnegative jumps (by taking $\sigma=0$ and $\nu\{(0,\infty)\}<\infty$) or BM (by taking $\nu\{(0,\infty)\}=0$). Importantly, however, our class of L\'evy processes also includes processes with $\mathbb{P}$-a.s.\ infinitely many jumps over finite time intervals (i.e., $\nu\{(0,\infty)\}=\infty$), such as the Gamma process and the inverse Gaussian process. For an account of some other examples, see e.g.\ \cite[Section 2.2]{Debicki2015} and \cite[Section 1.2]{Kyprianou2006}. In the sequel we denote by
\begin{align}\label{eq:defrho}
\varrho\equiv\mathbb{E}J(1)=\varphi'(0+)=c+\int_{(0,\infty)}y\,\nu\left({\rm d}y\right)=c+\int_0^\infty \nu\left\{(y,\infty)\right\}{\rm d}y\,
\end{align}
the process' asymptotic drift, which we throughout assume to be negative.
In addition, we let $V$ denote the initial push given to the process $J(\cdot)$, defined as a positive random variable that is independent from $J(\cdot)$. The resulting process is, for $t\geqslant 0$,
\begin{align}W(t)\equiv W_V(t)\equiv V+J(t),
\end{align}
which we consider between $0$ and the stopping time
\begin{equation}
    \tau\equiv\tau_V\equiv\inf\left\{t\geqslant 0: W_V(t)=0\right\}\,.
\end{equation}
where we note that under $\varrho<0$ and $\mathbb{E}V<\infty$ we have $\mathbb{E}\tau<\infty$ (see, e.g., \cite[Section 2]{Kella1991b}). 
We proceed by stating the main objectives of this paper. 
With Borel function  $f:\mathbb{R}_+\rightarrow\mathbb{R}_+$\color{black}, we define the following four functionals associated to the process $W(\cdot)$:
\begin{align}\label{eq: functional def}
    &A_{f}\equiv A_{f}(V)\equiv\int_0^\tau f\circ W(t)\,{\rm d}t\equiv\int_0^\tau f(W(t))\,{\rm d}t \ \  ,\\& D_{f}\equiv D_{f}(V)\equiv\sum_{t\in[0,\tau]: W(t-)<W(t)} {f}\circ W(t-)\equiv\sum_{t\in[0,\tau]: W(t-)<W(t)} {f}( W(t-))\ .\nonumber
\end{align}
Special cases include: (1)~if $f= 1$, then $A_{f}$ equals the stopping time $\tau$, (2)~if $f$ is the identity ${\rm Id}$, then $A_{f}$ equals the area under the graph of $J(\cdot)$ until $\tau$, (3)~if ${f}=1$, then $D_{{f}}$ counts the number of jumps of $J(\cdot)$ until $\tau$, (4)~if ${f}={\rm Id}$, then $D_{{f}}$ represents {the aggregate of the values of the left-limit of $W(\cdot)$ at the jump epochs happening until $\tau$}.

Our concrete research question concerns joint moments of the functionals $A_{f_i}$ and $D_{g_i}$, for given Borel functions ${f_1}, {f_2}, {g_1}, {g_2}$. More specifically, we wish to derive explicit expressions of the high-order moments of the random vector 
\begin{equation}\label{eq: the vector}
\left(A_{f_1},A_{f_2},D_{g_1},D_{g_2}\right)
\end{equation} 
in terms of (i)~the moments of $V$, (ii)~the moments corresponding to the measure $\nu(\cdot)$, and (iii)~the parameters $c$ and $\sigma$. Our findings demonstrate how this can be done for the case that $f_1$, $f_2$, $g_1$ and $g_2$ are all of a polynomial form, the main result being a recursive scheme via which the joint expectation \begin{align}\label{eq:HOJM}
   \mathbb{E}\left\{A_{f_1}^{k_1}A_{f_2}^{k_2}D_{g_1}^{\ell_1}D_{g_2}^{\ell_2}\right\}
\end{align} can be evaluated for any combination of nonnegative integers $(k_1,k_2,\ell_1,\ell_2)$ and polynomials $f_1$, $f_2$, $g_1$ and $g_2$. For instance when picking $f_1=1$ and $f_2={\rm Id}$, we obtain joint moments pertaining to the stopping time $\tau$ and the area under the graph of $J(\cdot)$ until $\tau.$

An important special case, particularly from a queueing-theory perspective, is the one in which the process $J(\cdot)$ is of CPP type and $V$ has the same distribution as the jumps of $J(\cdot)$. We stress, though, that our results by no means do not require this choice; see also Remark \ref{rem:R1}. \color{black}

\subsection{Existing literature}
Thus far, only in very specific cases explicit expressions (or computation schemes) have been developed for the high-order moments of \eqref{eq: the vector}. We proceed by providing a succinct overview, where we subsequently discuss the settings with  spectrally positive \color{black}  L\'evy input, CPP input, and BM input.

\subsubsection*{ Spectrally positive \color{black}  L\'evy input} In the case when $J(\cdot)$ is a L\'evy process with nonnegative jumps (i.e., a  spectrally positive \color{black}  L\'evy process), it is well-known \cite[Section 8.1]{Kyprianou2006} that the Laplace-Stieltjes transform (LST) of $\tau$ is given by
\begin{align}
{\mathbb E}\,e^{-\beta\tau} = {\mathbb E} \,e^{-\varphi^{-1}(\beta)\,V},
\end{align} with $\varphi^{-1}(\cdot)$ the right inverse of the Laplace exponent $\varphi(\cdot)$. Thus, conceptually, the high order moments of the hitting time $\tau$ (i.e., $A_f$ with $f\equiv 1$) can be computed via repeated differentiation of this LST. The paper \cite{Abundo2013} focuses on the characterization of the LST for a more general class of functionals. That is, consider some nonnegative real-valued measurable function $f(\cdot)$ and assume that $V$ equals to a constant $x\in(0,\infty)$. Then, \cite[Theorem 2.3]{Abundo2013} states that 
\begin{equation}
    M_f(\lambda)\equiv M_f(\lambda;x)\equiv \mathbb{E}\exp\left\{-\lambda\int_0^{\tau_x}f\left[x+J(t)\right]{\rm d}t\right\}\ \ , \ \ \lambda>0
\end{equation}
is a solution of certain functional equation (equipped with boundary conditions) which is phrased in terms of the generator of $J(\cdot)$. This characterization of $M_f(\lambda)$ was shown to be helpful in specific examples, e.g.: 
\begin{enumerate}
    \item When $\sigma>0$ and there are no jumps, $M_{\text{Id}}(\cdot)$ is a solution of a Schr\"odinger equation which can be expressed in terms of Airy functions. 

    \item When $\sigma=0$ and the jump process is a Poisson process with rate $\theta\in(0,1)$ (i.e., all jump sizes are equal to one), it was shown that
    \begin{equation*}
        M_{\text{Id}}(\lambda)=\frac{\theta^{[x]+1}}{\prod_{j=0}^{[x]}[\theta+\lambda x-j]}\ \ , \ \ \forall \lambda>0
    \end{equation*}
    where $[x]$ is the integer part of $x$. 
\end{enumerate}
We note that, as stated in \cite[Example 5]{Abundo2013}, the application of this approach to the more general case in which there are Brownian motions and jumps together will not yield any explicit solution even when considering the relatively simple functions $f\equiv1$ or $f\equiv\text{Id}$. 

\color{black}
\subsubsection*{Compound Poisson input} Assume that $J(\cdot)$ is a CPP with rate $\lambda\in(0,\infty)$, nonnegative jumps (that are distributed as the initial push $V$) minus a unit drift.  In addition, denote $N\equiv D_g$ with $g\equiv 1$, i.e., $N$ is equal to the number of jumps that $J(\cdot)$ has  during the time interval $(0,\tau)$. For the case that $f\equiv 1$ as well, \cite[Section II.4.4]{Cohen2012} provides the joint LST of $N+1$ and $\tau$ in terms of a fixed point relation. In a similar fashion, we can derive the following fixed point relation for the joint LST of $N$ and $\tau$: for $\alpha,\beta>0$, 
\begin{equation} \label{eq: lst equation}
\gamma(\alpha,\beta)\equiv\mathbb{E}\,e^{-\alpha N-\beta \tau}=
\,\color{black}{\mathbb E}\, e^{-\left\{\beta + \lambda\left[1-e^{-\alpha}\color{black}\gamma(\alpha,\beta)\right]\right\}\,V}.
\end{equation}
Suppose that we wish to compute $\mathbb{E}\tau^kN^\ell$.
To this end, apply the Fa\'a di Bruno's formula in combination with the general Leibniz rule in order to differentiate both sides  of \eqref{eq: lst equation}, $k$ times with respect to $\beta$, $\ell$ times with respect to $\alpha$, and then take $(\alpha,\beta)\downarrow(0,0)$. For every $k,\ell\in \mathbb{Z}_+$, this yields an equation from which one can isolate $\mathbb{E}\tau^kN^\ell$ and get an expression in terms of $\lambda$, the first $\ell+k$ moments of $G(\cdot)$ and $\mathbb{E}N^i\tau^j$ for $i\in\{1,\ldots \ell\}$, $j\in\{1,\ldots ,k\}$. As a result, we can recursively determine all the joint moments of $N$ and $\tau$. 

Consider the case where $V$ has the distribution $G_V(\cdot)$ which is not the same as the jump distribution of $J(\cdot)$. In this case, denote the time until $W_V(\cdot)$ hits the origin by $\tau_V$ and the number of jumps it has until $\tau_V$ by $N_V$. In this case, we know that given $V$, the number of records in the (decreasing) ladder process until $\tau$ has a Poisson distribution with rate $\lambda V$. Thus, by the strong Markov property of $J(\cdot)$ we shall apply the LST of a compound Poisson distribution in order to get,  for $\alpha,\beta>0$,
\begin{align}\label{eq: LST CPP!}
    \gamma(\alpha,\beta;G_V)&\equiv\mathbb{E}e^{-\alpha N_V+\beta\tau_V}\\&\nonumber=\mathbb{E}\exp\left\{-V\left[\beta+\lambda\left(1-\mathbb{E}e^{-[\alpha (N+1)+\beta\tau]}\right)\right]\right\}=e^{-V\left\{\beta+\lambda\left[1-e^{-\alpha}\gamma(\alpha,\beta)\right]\right\}}\ \ .
\end{align}
This means that the high-order moments of the random vector $(N_V,\tau_V)$ will all be linear combinations of the moments of $V$ with coefficients determined by the high-order moments of $(N,\tau)$.

Complications arise, however, when considering functionals of $W(\cdot)$ for which we lack a fixed-point relation for the LST, entailing that one cannot apply the procedure described above to recursively compute the corresponding moments. For example, in the case that $f(\cdot)={\rm Id}$, in which $A_f$ represents the area beneath the sample path of $W(\cdot)$ until $\tau$, no fixed-point relation for the relevant LST is known. 
Following other techniques, a few partial results were found though. More specifically, with $\mu_i$ denoting the $i$-th moment of the jump distribution of $J(\cdot)$, and $f(\cdot)={\rm Id}$,
\begin{enumerate}
    \item Iglehart \cite{Iglehart1971} showed that, applying the time normalization $c=1$,
\begin{equation}\label{eq: Iglehart1971}
\mathbb{E}A_f=\frac{\mu_2}{2(1-\rho)^2}\,,
\end{equation}
where $\rho\equiv 1+\mathbb{E}J(1)=\lambda\mathbb{E}V<1$.
In addition, he stated the challenge of finding the other moments of $A_f$ as an open problem.

\item Cohen \cite{Cohen1978} applied a level-crossing technique to derive that
\begin{equation}\label{eq: Cohen1978}
    \mathbb{E}A_f^2=\frac{\mu_4}{4(1-\rho)^3}+\frac{4\lambda\mu_2\mu_3}{3(1-\rho)^4}+\frac{5\lambda^2\mu_2^3}{\color{black}4\color{black}(1-\rho)^5}\,.
\end{equation}
\end{enumerate}
Concerning moments of order greater than two, to our knowledge, the only relevant existing references are:
\begin{enumerate}
    \item  Borovkov, Boxma and Palmowski \cite{Borovkov2003}, who showed that when the jump distribution of $J(\cdot)$ is exponential, the LST of $A_f$ (with $f(\cdot)$ still the identity ${\rm Id}$) can be written in terms of Whittaker functions.

    \item As mentioned above, Abundo \cite{Abundo2013} assumed that the jumps seize are all equal to one, and provided an explicit formula of the LST of $A_f$ (with $f(\cdot)$ still the identity ${\rm Id}$). 
\end{enumerate} 
Thus, even when considering the simpler setup with $J(\cdot)$ being a CPP, the computation of $\mathbb{E}A_f^k$ for $k=3,4,\ldots$ remained an open problem. 

Another related work concerning the instance that $J(\cdot)$ is a CPP with nonnegative jumps, is by Glynn \cite{Glynn1994}, relating  $\mathbb{E}D_g(x)$ to an associated Poisson's equation. A consequence of the findings in \cite{Glynn1994} is that when $g(\cdot)$ is a polynomial, then $\mathbb{E}D_g(x)$ is a polynomial in $x$. To the best of our knowledge, there is no additional literature regarding $\mathbb{E}D_g^\ell(x)$ when $\ell\geqslant 2$.

\subsubsection*{Brownian motion input} We proceed by considering the case where $V\equiv x$ for some $x\in(0,\infty)$ and $J(t)\equiv ct+B(t)$ for every $t\geq0$ such that $B(\cdot)$ is a standard BM and $c<0$.  Abundo and Del Vescovo \cite{Abundo2017} derived a computational scheme to derive the precise expression of the joint moments 
\begin{equation}
    I_{k,\ell}(x)\equiv\mathbb{E}\tau_x^k\left(\int_0^{\tau_x}\left[x-J(t)\right]{\rm d}t\right)^\ell\ \ , \ \ k,\ell\in\mathbb{Z}_+\,.
\end{equation}
They showed how this scheme produced the expressions of $I_{k,\ell}(x)$ for small values of $k$ and $\ell$. Furthermore, they proved \cite[Proposition 2.5]{Abundo2017} that $I_{k,\ell}(x)$ is a polynomial in $x$ with a degree equal to $k+2\ell$. Note that when $V$ is a positive random variable independent of $B(\cdot)$ and has a cdf $G$, $\int_0^\infty I_{k,\ell}(x){\rm d}G(x)$ is a linear combination of the first $k+2\ell$ moments of $V$.  

More results on the joint distribution of the first passage time $\tau_x$ and the so-called \textit{first passage area} $\int_0^{\tau_x}\left[x-J(t)\right]{\rm d}t$ appear in several works motivated by certain applications in physics (see, e.g., \cite{Kearney2005, Kearney2007, Kearney2014} and the references therein).
\color{black}
\subsection{Motivations} The computation of high-order joint moments of the type \eqref{eq:HOJM} has applications across various domains. In this subsection we subsequently point out how our results are relevant in the context of L\'evy fluctuation theory, in the context of insurance/risk models, and in a queueing setting.
 
 \subsubsection*{L\'evy processes with secondary jump input} 
 Consider a probability space $(\Omega,\mathcal{F},\mathbb{P})$ equipped with a filtration $\mathbb{F}\equiv(\mathcal{F}_t)$ which is right-continuous increasing sequence of ${\mathbb P}$-complete sub-$\sigma$-fields of $\mathcal{F}$. In addition, assume that $X(\cdot)\equiv\{X(t):t\geqslant 0\}$ is the right-continuous version of a L\'evy process with respect to $\mathbb{F}$, i.e., it is adapted to $\mathbb{F}$, and for every $0<s<t$ we have that $X(t-s)$ is distributed as $X(t)-X(s)$ and $X(t)-X(s)$ is independent of $\mathcal{F}_s$. Let $(\tau_n)_{n=0}^\infty$ be a strictly increasing sequence of stopping times in $\mathbb{F}$ such that $\tau_0\equiv0$, and define a counting process
 \begin{equation}
     N(t)\equiv\sup\left\{n\in{\mathbb N}:\, \tau_n\leqslant t\right\}\, ,
 \end{equation}
 for $t\geqslant 0$.
If $(U_n)_{n=0}^\infty$ is a sequence of random variables adapted to the filtration $(\mathcal{F}_{\tau_n})_{n=0}^\infty$, then a general description of a {\it L\'evy processes with secondary jump input} is given via
\begin{equation}
Y(t)\equiv X(t)+\sum_{n=0}^{N(t)}U_n\ ,   
\end{equation}
 for $t\geqslant 0$.
For more details regarding this class of processes and their applications in the applied probability literature, see e.g.\ the seminal papers \cite{Kella1991a,Kella1991b} and some consecutive works such as~\cite{Jacobovic2019,Jacobovic2021,Kella1998,Kella2023}.
 
Now, consider a regenerative process $\left\{Y(t):t\geqslant 0\right\}$ with independent cycles that are distributed as $\left\{W_V(t):0\leqslant t<\tau\right\}$. The thus constructed process is a special case of a L\'evy processes with secondary jump input, as introduced in the preceding paragraph. Since the distribution of $\tau$ is non-arithmetic, we can use the classical theory regarding regenerative processes; see, e.g., \cite[Chapter VI]{Asmussen2003} for a detailed account. It specifically implies that \color{black} as $\varrho<0$\color{black}, there exists a random variable $Y^*$ such that $Y(t)\rightarrow_{\rm d}Y^*$ 
 as $t\to\infty$ and 
 \begin{equation}
     \mathbb{E}f\left(Y^*\right)=\frac{\mathbb{E}A_f}{\mathbb{E}\tau}\,,
 \end{equation}
 for any nonnegative Borel function $f(\cdot)$. 
 A popular approach to quantify the steady-state mean $\mathbb{E}f\left(Y^*\right)$ is via stochastic simulation, where a frequently used class of algorithms is of the \textit{regenerative} type; see, e.g., \cite[Chapter IV, Section 4]{Asmussen2007}. In the study of the probabilistic properties of the resulting estimators, one greatly benefits from the availability of high-order moments of $\left(A_f,\tau\right)$
 (e.g.\ when constructing confidence intervals, or upper bounds on the tail distribution).  

\subsubsection*{Insurance/risk model}
For a given $\theta\geqslant 0$, we define the process
\begin{equation}
    X_\theta(t)\equiv \theta+rt-J(t)\  , 
\end{equation}
for $t\geqslant 0$, where $J(\cdot)$ is a subordinator (for simplicity assumed to have zero drift) and $r\in(0,\infty)$ such that $\mathbb{E}J(1)<r$. In particular, when $J(\cdot)$ is a CPP with nonnegative jumps, the resulting process describes the evolution of the surplus level of an insurance company in the classical {\it Cram\'er-Lundberg model} \cite{Asmussen2010,Mandjes2023}.

Let $\tau_\theta$ be the {\it ruin time}, i.e., the first time when the surplus process $X_\theta(\cdot)$ becomes negative. A traditional objective in ruin theory concerns the evaluation of the ruin probability $\mathbb{P}\left\{\tau_\theta<\infty\right\}$, but recently attention shifted to the analysis of more general classes of functionals of $X_\theta(\cdot)$ conditionally on the event $\{\tau_\theta<\infty\}$. One could for instance consider functionals of the type
\begin{equation}\label{eq: cost-risk}
    C_\theta(h)\equiv\int_0^{\zeta_\theta}h\left[-X_\theta\left(\tau_\theta+t\right)\right]{\rm d}t
\end{equation}
for some nondecreasing right-continuous function $h:\mathbb{R}_+\rightarrow\mathbb{R}_+$; here
    \begin{equation}
        \zeta_\theta\equiv\inf\left\{t> 0:X_\theta\left(\tau_\theta+t\right)=0\right\},
    \end{equation}
so that $\zeta_\theta$ defines the length of the interval that starts at time $\tau_\theta$ and that ends at the first time the surplus level hits a nonnegative value again. 
Two examples of relevant functionals $h$ are: 
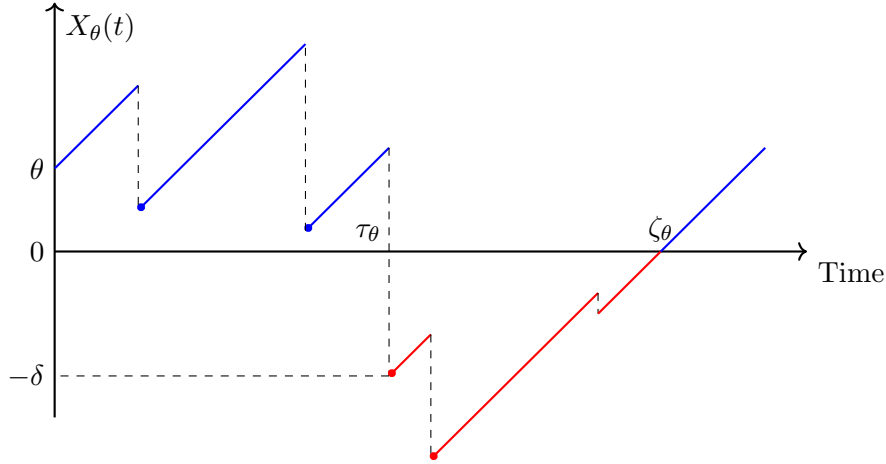
\begin{figure}
    \centering 
\begin{tikzpicture}[scale=1.1]
\draw[thick,->] (0,2) -- (9,2) node[anchor=north west]{$\text{Time}$};
\draw[thick,->] (0,0) -- (0,5) node[anchor=north west]{$X_\theta(t)$};

\draw (0,2) coordinate  node [anchor=east] {$0$};
\draw (0,3) coordinate  node [anchor=east] {$\theta$};

\draw [line width=0.8pt, blue] (0,3) -- (1,4);
\draw [{Circle[length=3pt]}-{},line width=0.8pt,blue](1,2.5)--(3,4.5);
\draw[dashed] (1,4)--(1,2.5);
\draw [{Circle[length=3pt]}-{},line width=0.8pt,blue](3,2.25)--(4,3.25);
\draw[dashed] (3,2.25)--(3,4.5);

\draw[dashed] (4,3.25)--(4,0.5);

\draw [{Circle[length=3pt]}-{},line width=0.8pt,red](4,0.5)--(4.5,1);
\draw[dashed] (4.5,1)--(4.5,-0.5);

\draw [{Circle[length=3pt]}-{},line width=0.8pt,red](4.5,-0.5)--(6.5,1.5);

\draw[dashed] (4,0.5)--(0,0.5);
\draw (0,0.5) coordinate  node [anchor=east] {$-\delta$};

\draw[dashed] (6.5,1.5)--(6.5,1.25);

\draw [line width=0.8pt,red](6.5,1.25)--(7.25,2);

\draw [line width=0.8pt,blue](7.25,2)--(8.5,3.25);

\draw (3.75,2) coordinate  node [anchor=south] {$\tau_\theta$};

\draw (7.25,2) coordinate  node [anchor=south] {$\zeta_\theta$};

\end{tikzpicture}
\caption{An illustration of $X_\theta(\cdot)$ when $r=1$ and $J(\cdot)$ is a CPP with nonnegative jumps. Note that conditionally on an  undershoot which equals $\delta$, the mirror image of $\left\{X_\theta(t):\tau_\theta\leqslant t\leqslant \zeta_\theta\right\}$ with respect to the horizontal axis is distributed as $\delta+J(t)-t$, revealing the relation to the current work.}\end{figure}

\begin{enumerate}
    \item The choice $h\equiv{1}$ yields $\zeta_\theta$ itself, in the literature referred to as the so called \textit{time to recovery}. This quantity was analyzed by e.g.\ Dos Reis \cite{Dos Reis1993,Dos Reis2000} in the framework of the classical Cram\'er-Lundberg model. \smallskip

    \item The choice $h(x)\equiv x$  yields the so called \textit{area in red}. This functional was analyzed by Picard \cite{Picard1994} in the Cram\'er-Lundberg model. Later, Lkabous and Wang \cite{Lkabous2023} derived the first moment of the area in red in a more general L\'evy-driven setup in terms of scale functions \cite[Chapter 8]{Kyprianou2006}.  
\end{enumerate}
We proceed by illustrating how the current results can be used to analyze the quantities defined above. 
Choosing $h(x)\equiv x^k$ for some $k\in\mathbb{Z}_+$, we have, for any $\ell\in{\mathbb N}$,
\begin{equation}\label{eq: risk moment}
\mu_{\theta}(k,\ell)\equiv\mathbb{E} C_\theta^\ell(h)=\mathbb{E}\left\{\left[\int_0^{\zeta_\theta-\tau_\theta}\left[-X_\theta\left(\tau_\theta+t\right)\right]^k{\rm d}t\right]^\ell\bigg|\ \tau_\theta<\infty\right\}.
\end{equation}
Then, as illustrated in Figure~1, conditioning and un-conditioning on the \textit{undershoot} $\Delta_\theta\equiv -X_\theta(\tau_\theta)$, combined with the strong Markov property of the driving L\'evy process, yields
\begin{equation}\label{eq: area in red}
    \mu_\theta(k,\ell)=\int_{(0,\infty)}\mathbb{E}\left\{\left[\int_0^{\widetilde{\tau}_\delta}\left[\delta+ J(t)-rt\right]^k{\rm d}t\right]\right\}^\ell\mathbb{P}\left\{\Delta_\theta\in \rm d\delta\right\}\,,
\end{equation}
where, for $\delta>0$,
\begin{equation}
    \widetilde{\tau}_\delta\equiv\inf\left\{t\geqslant 0:\delta+J(t)-rt=0\right\}\ .
\end{equation}
This reveals the relation to the current work in which we provide an expression of the expectation inside the integral in \eqref{eq: area in red}. In our paper we show that this expression is a polynomial in $\delta$, i.e., in order to evaluate $\mu_\theta(k,\ell)$ it is left to derive the moments of the distribution of $\Delta_\theta$. For example, when $J(\cdot)$ is a CPP with nonnegative jumps, the transform which appears in \cite[Exercise 1.2]{Mandjes2023} might be particularly useful for that purpose. 

\subsubsection*{M/G/1 queue}
The {\it Skorohod reflection} \cite[Chapter IX, Section 2]{Asmussen2003} of $W_V(\cdot)$ is defined, for $t\geqslant 0$, by
\begin{equation}\label{SKOR}
    \overline{W}(t)\equiv \overline{W}_V(t)\equiv W_V(t)-\inf_{0\leqslant s\leqslant t}\left[W_V(s)\right]^-,
\end{equation}
where $a^-\equiv\min\{-a,0\}$ for any $a\in\mathbb{R}$. 

Consider the case that $J(\cdot)$ is a CPP with nonnegative jumps minus a unit drift. In addition, let $V$ be a random variable which is independent of $J(\cdot)$ and distributed as the jumps of $J(\cdot)$. In that case $\overline{W}_x(\cdot)$ describes the evolution of the workload in an M/G/1 queue  with an initial workload level of~$x$. A detailed account of this type of storage systems can be found in, e.g., \cite[Chapters VIII \& X]{Asmussen2003}, \cite{Cohen2012} and \cite{Prabhu1998}. Note, that the process $\overline{W}_0(\cdot)$ stochastically regenerates whenever it hits the origin and the regeneration epochs have a non-arithmetic distribution. Therefore, by the regenerative-process theory, as found in  e.g.\ \cite[Chapter VI]{Asmussen2003}, under the evident stability condition, there exists a random variable $W^*$ such that $\overline{W}_0(t)\rightarrow_{\rm d}W^*$ as $t\to\infty$.

Let $N(\cdot)$ be the arrival process of the jumps in ${J}(\cdot)$, and suppose we observe the process $\overline{W}_0(\cdot)$ according to $N(\cdot)$ during the time window $[0,T]$. When trying to estimate $\mathbb{E}\,g(W^*)$ for some Borel function $g:\mathbb{R}_+\rightarrow \mathbb{R}_+$, we can use ideas similar to the ones we  discussed above in the context of L\'evy processes with secondary jumps. By the {\sc pasta} property \cite[Theorem 1]{Wolf1982},
\begin{equation}\label{eq: renewal reward}
\mathbb{E}\,g(W^*)\equiv\frac{\mathbb{E}D_g}{\mathbb{E}D_{\textbf{1}}}\,.   
\end{equation}
Equation \eqref{eq: renewal reward} suggests that a consistent estimator can be constructed via the regenerative method \cite[Section IV.4]{Asmussen2007}. To assess the statistical properties of this estimator (in the regime that $T\to\infty$), such as its asymptotic variance and asymptotic tail probabilities, it is of great help to have expression for the high-order moments of $\left(D_g,D_\textbf{1}\right)$.

\subsection{Organization}  

The remainder of this paper is organized as follows.
Section~\ref{sec: first moment} presents the results about the computation of the {\it first moment} of \eqref{eq: the vector}, which leads to recursive computation schemes when $f(\cdot)$ is polynomial. Then, Section~\ref{sec: high-order moments} points out how to utilize the Markov property of the underlying L\'evy process, jointly with the results regarding the first moment of \eqref{eq: the vector}, so as to derive {\it high-order moments} of \eqref{eq: the vector}. Section~\ref{sec: cost process} includes some applications of the results from the former sections, aiming to prove {\it structural properties} of the functional $A_f(x)$, to be interpreted as a stochastic process in the initial level~$x$. 

In Section~\ref{sec: transform} we consider the {\it area beneath the workload process} in an M/G/1 queue until an independent exponentially distributed time $T_\beta$ with mean $\beta^{-1}$. It is explained how its LST can  expressed in terms of the joint LST of $\left(A_{\rm Id},\tau\right)$  (with ${\rm Id}$ denoting the identity function on $\mathbb{R}_+$). This translation formula, relating the area until hitting time $\tau$ to its counterpart until $T_\beta$,  is particularly useful, as we have access to the joint LST of $A_{\rm Id}$ and $\tau$ via the expressions for the joint moments presented in Section~\ref{sec: high-order moments}.

Section~\ref{sec: demonstration} demonstrates how the existing results in \cite{Cohen1978,Iglehart1971}, as given by \eqref{eq: Iglehart1971}--\eqref{eq: Cohen1978}, are easily reproduced by applying our results, and how explicit expression for other (joint) moments can be found. Finally, Section \ref{sec: open problems} presents a collection of open problems related to the research presented in this paper.

\section{First moment}\label{sec: first moment}
In this section, we first derive an expression for the first moment ${\mathbb E}A_f$ in Theorem \ref{thm: basic}, and use this to characterize its counterpart ${\mathbb E}D_f$ in Theorem \ref{thm: D_f}. In the second subsection we point out how these expressions can be made explicit in the case of polynomial $f$, with Theorem \ref{thm: Takacs} providing a recursive algorithm. 

\subsection{General expressions.}
In our analysis we denote, for each $i\in{\mathbb Z}_+$, with $G_V(\cdot)$ the cumulative distribution function of the initial push $V$,
\begin{equation}
\mu_i\equiv\int_0^\infty y^i\,{\rm d}G_V(y)\,.
\end{equation}
We in addition  observe that $\varrho$, as defined in \eqref{eq:defrho}, is negative and finite. As follows from, e.g., \cite[Section 2]{Kella1991b}, we thus have that the expected cycle length, given by
\begin{equation}\label{eq: tau expectation}
\mathbb{E}\tau=-\frac{\mu_1}{\varrho},
\end{equation}
is finite. The following theorem provides a constructive method to compute the first moment $\mathbb{E}A_f$ for a given Borel function $f$.

\begin{theorem}\label{thm: basic}
    Assume that $U$ and $\zeta$ are nonnegative random variables such that
    \begin{enumerate}
        \item $V$, $U$ and $\zeta$ are independent.

        \item $U$ is distributed uniformly on $[0,1]$.

        \item The LST of $\zeta$ is given by
        \begin{equation}\label{psi}
    \psi(\alpha)\equiv\mathbb{E}e^ {-\alpha\zeta}=\frac{\alpha\varphi'(0+)}{\varphi(\alpha)}\ \ , \ \ \alpha> 0\,.    
    \end{equation}
    
    \end{enumerate} Then, for any Borel function $f:\mathbb{R}_+\rightarrow\mathbb{R}_+$, 
\begin{equation}
    \mathbb{E}A_f= -\frac{1}{\varrho}\int_0^\infty x\,\mathbb{E}f\left(\zeta+xU\right){\rm d}G_V(x)\,.
\end{equation}
\end{theorem}

\begin{proof}
First it is observed that it suffices to provide the proof under the assumption that $V=x$, $\mathbb{P}$-a.s.\ for a given $x>0$. Then, the claim follows via an integration with respect to ${\rm d} G_V(\cdot)$.  

Since $\mathbb{E}\tau=-{x}/{\varrho}<\infty$, there is a probability measure on $\left(\mathbb{R}_{+},\mathcal{B}(\mathbb{R}_+)\right)$ which is defined by, for any given $E\in\mathcal{B}(\mathbb{R}_+)$,
\begin{equation}\label{eq: P stationary}
\mathbb{P}^*\{E\}
\equiv\frac{1}{\mathbb{E}\tau}\,\mathbb{E}\Big[\int_0^\tau\textbf{1}_{\left\{W(t)\in E\right\}}{\rm d}t\Big]\ \ .
\end{equation} 
In addition, we note that, with $\mathbb{E}^*$ denoting expectation with respect to $\mathbb{P}^*$,
\begin{equation}\label{eq: A_f1}
\mathbb{E}A_f=\mathbb{E}\Big[\int_0^\tau f\circ W(t)\,{\rm d}t\Big]=-\frac{x\, \mathbb{E}^*f}{\varrho}\,.
\end{equation}
Now, let $\xi$ be a random variable such that $\mathbb{P}\left\{\xi\in E\right\}=\mathbb{P}^*\left\{E\right\}$ for every $E\in\mathcal{B}\left(\mathbb{R}_+\right)$. Then \cite[Theorem 5.1]{Kella1991a} implies that, for any $\alpha>0$,
\begin{equation}
\mathbb{E}e^{-\alpha\xi}=\frac{\alpha\varphi'(0+)}{\varphi(\alpha)}\cdot\frac{1-e^{-\alpha x}}{\alpha x}\,.
\end{equation}
Hence, $\xi$ 
is distributed as the sum of a `Pollaczek–Khinchine distributed random variable' (having LST ${\alpha\varphi'(0+)}/{\varphi(\alpha)}$, that is; see e.g.\ \cite[Theorem 4.2]{Kella1991a}) and an independently sampled uniformly distributed random variable on $[0,x]$.
This finishes the proof. 
\end{proof}

\begin{remark}\normalfont \label{rem:R1}
The proof of Theorem \ref{thm: basic} reveals that we can actually work with the initial push $V$ being a given, deterministic level: if $V=x>0$, $\mathbb{P}$-a.s., then 
\begin{equation}
    \mathbb{E}A_f= -\frac{1}{\varrho}x\,\mathbb{E}f\left(\zeta+xU\right)\,.
\end{equation}
Along the same lines, we wish to emphasize that all main results presented in this paper thus cover the case of starting at a given, deterministic level $x$ (the exception being the M/G/1-specific results of Section \ref{sec: transform}). It is noted though that \eqref{eq: lst equation}, \eqref{eq: Iglehart1971} and \eqref{eq: Cohen1978}, all  results from the existing literature that we discussed in the M/G/1 context, require that the initial level be distributed as the jumps of the driving compound Poisson process. \color{black}
\end{remark}

\begin{remark}\normalfont
 In \eqref{psi}, $\psi(\cdot)$ is expressed in terms of the generalized Pollaczek–Khinchine formula \cite[Theorem 4.2]{Kella1991a}. In fact, if the driving process is a CPP, then the density of $\zeta$ can be written as the solution to an integral equation \cite[Subsection 1.7]{Brill2008}. \end{remark}
 
\begin{remark}
    \normalfont  Consider a regenerative process $Y(\cdot)$ in which the cycles have the distribution of  $\{W(t):0\leqslant t<\tau\}$. This process has a finite expected cycle length, so that the probability measure $\mathbb{P}^*$ which was defined in \eqref{eq: P stationary} has multiple appealing interpretations. (1)
    The renewal reward theorem implies that for every set $E\in\mathcal{B}\left(\mathbb{R}\right)$, $\mathbb{P}^*$ assigns to $E$ the long-run average time that the process $Y(\cdot)$ attains values in $E$. (2) $\mathbb{P}^*$ is the  distribution in a stationary version of $Y(\cdot)$,  as constructed in \cite{Thorisson1992}. (3) Since $\tau$ has a non-arithmetic distribution, the probability measure $\mathbb{P}^*$ also defines the limiting distribution of $Y(t)$ as $t\to\infty$ \cite[Chapter VI]{Asmussen2003}.
\end{remark}

The second main result of this section concerns the relation between the first moment of $A_f$ and the first moment of $D_f$. When the jump intensity $\lambda\equiv\nu\{(0,\infty)\}$ is finite, it follows immediately from the {\sc pasta} property \cite{Wolf1982}. In the case this jump intensity is infinite, however, our proof reveals that a considerably more delicate argumentation is needed.
\begin{theorem}\label{thm: D_f}
Let $f(\cdot)$ be a nonnegative Borel function on $\mathbb{R}_+$. Then, 
\begin{equation}\label{eq:Little-like}
    \mathbb{E}D_f=\lambda\,\mathbb{E}A_f
    \end{equation}
    with the convention that $0\cdot\infty=0$. 
\end{theorem}
Remark that in the case of drifted Brownian motion, as discussed in the introduction, we have that, in \eqref{eq:Little-like}, $\lambda$ is to be taken equal to $0$. 

The proof of Theorem \ref{thm: D_f} (as well as proof of other claims in this section) is based on the L\'evy-It\^{o} decomposition \cite[Theorem 2.1]{Kyprianou2006} which states that $J(\cdot)$ can be expressed as the superposition of four independent processes: 
\begin{enumerate}
    \item A deterministic drift process  $t\mapsto ct$ with $c<0$.
    
    \item $B_{\sigma}(\cdot)\equiv\sigma B(\cdot)$ which is a driftless Brownian motion with a volatility  $\sigma\geqslant0$.

    \item $\Psi_0(\cdot)$ which is a CPP with an arrival rate $\lambda_0\equiv\nu\left\{[1,\infty)\right\}$ and jump distribution ${\lambda_0}^{-1}\,{\nu({\rm d}y)}$.

    \item A sum of a sequence of independent processes $\Psi_1(\cdot),\Psi_2(\cdot),\ldots$ such that for each $i\in{\mathbb N}$, $\Psi_i(\cdot)$ is a CPP which has an arrival rate $\lambda_i\equiv\nu\left\{[2^{-(i+1)},2^{-i})\right\}$ and jump distribution ${\lambda_i}^{-1}\,{\nu({\rm d}y)}$.
\end{enumerate}
Note that in the above-mentioned description of the decomposition, we apply the convention that whenever $\lambda_i=0$, then $\Psi_i(\cdot)\equiv0$, $\mathbb{P}$-a.s.
In addition, $\lambda\equiv \sum_{i=0}^\infty\lambda_i$. 

\color{black} The proof of Theorem \ref{thm: D_f} is based on the main result of \cite{Wolf1982}, a classical finding in the applied probability literature that has become known as the \textit{Poisson arrivals see times averages} ({\sc pasta}) property. Morally, the {\sc pasta} result provides us with conditions under which a system's stationary distribution at Poisson epochs coincides with the system's long-run distribution.  For convenience and completeness, we are now providing its precise statement. 
\begin{theorem}[{\sc pasta}] Consider a probability space $(\Omega,\mathcal{F},{\mathbb P})$ with the following objects:
\begin{itemize}
    \item[$\circ$] $\mathbb{F}\equiv(\mathcal{F}_t)_{t\geqslant 0}$ is an increasing right-continuous sequence of ${\mathbb P}$-complete sub-$\sigma$-fields of $\mathcal{F}$.
    
    \item[$\circ$] $X(\cdot)\equiv\left\{X(t):t\geqslant 0\right\}$ is a c\`adl\`ag stochastic process.

    \item[$\circ$] $N(\cdot)\equiv \{N(t):t\geqslant 0\}$ is a simple c\`adl\`ag counting process, which is adapted to $\mathbb{F}$ and which has stationary increments.

    \item[$\circ$] For any $t,h\geqslant 0$, the increment $N(t+h)-N(t)$ is independent of $\mathcal{F}_t$.
\end{itemize}
For every $n\geqslant 0$ denote 
\begin{equation}
    T_n\equiv \inf\left\{t\in{\mathbb N}: N(t)\geqslant n\right\}.
\end{equation}
Then, the (continuous-time) process $\{X(t):t\geqslant 0\}$ and the (discrete-time) process $\{X(T_n-):n=0,1,2,\ldots\}$ have the same ergodic distribution and hence for any nonnegative Borel function $f(\cdot)$, the limits 
\begin{equation*}
    \lim_{t\to\infty}\frac{1}{t}\int_0^tf\circ X(s)\,{\rm d}s\ \ , \ \ \lim_{n\to\infty}\frac{1}{n}\sum_{i=1}^nf\circ X\left(T_n-\right)\ \ , \ \ \lim_{t\to\infty}\frac{1}{N(t)}\int_0^t f\circ X(s-)\,{\rm d}N(s)\,,
\end{equation*}
exist (possibly infinite) and they are equal $\mathbb{P}$-a.s. 
\end{theorem}
\color{black} In what follows, we apply this theorem with $X(\cdot)$ corresponding to a compound Poisson process with nonnegative jumps, and $N(\cdot)$ corresponding to the underlying Poisson process that defines the jumps epochs in $X(\cdot)$. Now we are ready to provide the proof of Theorem \ref{thm: D_f}.
\color{black}
\begin{proof}
    If $\lambda=0$, the result is trivial, so that it is left to provide a proof for the case $\lambda>0$ (possibly infinite). For any $i\in{\mathbb N}_0$, let $N_i(\cdot)$ be the arrival process corresponding to the jumps in $\Psi_i(\cdot)$, i.e., whenever $\lambda_i>0$, $N_i(\cdot)$ is a Poisson process with rate $\lambda_i$ and otherwise it equals zero $\mathbb{P}$-a.s. This entails that the superposition
    \begin{equation}\label{eq: N def}
    N(t)\equiv\sum_{i=0}^\infty N_i(t)=\sum_{i\in{\mathbb N}_0:\lambda_i>0}N_i(t)\ \ , \ \ t\geqslant0
    \end{equation}
   coincides with the arrival process of the jumps in $J(\cdot)$.

    Now, let $\{Y(t):t\geqslant0\}$ be a regenerative process with independent cycles that are distributed as $\{W(t):0\leqslant t<\tau\}$ such that $Y(t)=W(t)$ for every $0\leqslant t<\tau$ and note that $Y(\cdot)$ is c\`adl\`ag. 
    Then, as $\tau$ is positive and has a finite mean, the ergodic theorem for regenerative processes implies that
\begin{equation}
   \frac{1}{t}\int_0^tf\circ Y(s)\,{\rm d}s\xrightarrow{t\rightarrow\infty}\frac{\mathbb{E}A_f}{\mathbb{E}\tau}\ \ , \ \ \mathbb{P}\text{-a.s.}
\end{equation}
As a result, due to the {\sc pasta} property \cite{Wolf1982}, for any $i\in{\mathbb N}_0$ such that $\lambda_i>0$,
\begin{equation}
    \frac{1}{N_i(t)}\int_0^tf\circ Y(s-)\,{\rm d}N_i(s)\xrightarrow{t\rightarrow\infty}\frac{\mathbb{E}A_f}{\mathbb{E}\tau}\ \ , \ \ \mathbb{P}\text{-a.s.}
\end{equation}
In addition, for each $i\in{\mathbb N}_0$ such that $\lambda_i>0$, it is known that $N_i(t)/(\lambda_i t)\to 1$ as $t\to\infty$, $\mathbb{P}$-a.s. (by the law of large numbers for renewal processes), and hence for that case the renewal reward theorem leads to
\begin{align*}
    \frac{t}{N_i(t)}\cdot\frac{1}{t}\int_0^tf\circ Y(s-)\,{\rm d}N_i(s)&\xrightarrow{t\rightarrow\infty}\frac{1}{\lambda_i\,\mathbb{E}\tau}\mathbb{E}\int_0^\tau f\circ Y(s-)\,{\rm d}N_i(s)\\&=\nonumber\frac{1}{\lambda_i\,\mathbb{E}\tau}\mathbb{E}\int_0^\tau f\circ W(s-)\,{\rm d}N_i(s)\ \ , \ \ \mathbb{P}\text{-a.s.}
\end{align*}
Thus, because this limit is unique, we conclude that, for any $i\in{\mathbb N}_0$,
\begin{equation}\label{eq: D_g i}
\mathbb{E}\int_0^\tau f\circ W(t-)\,{\rm d}N_i(t)=\lambda_i\,\mathbb{E}A_f\,.
\end{equation}
Then, a summation of both sides of \eqref{eq: D_g i} over all indices $i\in{\mathbb N}$ such that $\lambda_i>0$ yields
\begin{equation}\label{eq: D_f bounded}
    \mathbb{E}D_f=\sum_{i=0}^\infty 
\mathbb{E}\int_0^\tau f\circ W(t-)\,{\rm d}N_i(t)=\lambda\,\mathbb{E}A_f\,.
\end{equation}
This completes the proof. 
\end{proof}

\subsection{Polynomial functionals}
The first moment ${\mathbb E}A_f$, as identified in Theorem \ref{thm: basic}, can be made explicit when the functional $f(\cdot)$ is a polynomial. The expression for $f(\cdot)={\mathcal P}_k(\cdot)$ with $\mathcal{P}_k(y)\equiv y^k$ for $k\in {\mathbb Z}_+$ and  $y>0$, is presented in
the next corollary.  \color{black} Recall the objects $\varphi(\cdot)$ and $\psi(\cdot)$ that were defined in respectively \eqref{eq: varphi} and \eqref{psi}. \color{black} For simplicity of notation, for any $k\in \mathbb{Z}_+$, we define 
\begin{equation}
    \psi_k\equiv\psi^ {(k)}(0+)\equiv\left.\frac{{\rm d}^k\psi(\alpha)}{{\rm d}\alpha^k}\right|_{\alpha\downarrow0},\:\:\:\:\:\varphi_k\equiv\varphi^ {(k)}(0+)\equiv\left.\frac{{\rm d}^k\varphi(\alpha)}{{\rm d}\alpha^k}\right|_{\alpha\downarrow0}.
\end{equation}

\begin{corollary}\label{cor: polynomial}
For any $k\in\mathbb{Z}_+$, 
\begin{equation}
    \mathbb{E}A_{\mathcal{P}_k}=-\frac{1}{\varrho}\sum_{i=0}^k\binom{k}{i}\frac{\mu_{k-i+1}}{k-i+1}(-1)^i\psi_i\,.
\end{equation}
\end{corollary}
\begin{proof}
    The proof follows immediately by inserting $f\equiv \mathcal{P}_k$ into the expression given in Theorem \ref{thm: basic}, in combination with a straightforward application of the binomial theorem. 
\end{proof}

From a practical standpoint it is clear that, in order to apply Corollary \ref{cor: polynomial}, it is necessary to know how to compute the first $k$ right-derivatives of $\psi(\cdot)$ at zero. Our model primitives, however, relate to right-derivatives of $\varphi(\cdot)$ at zero; recall our ambition to express the moments of interest in terms of the moments of $V$, the moments associated to $\nu(\cdot)$ and the parameters $c$ and $\sigma^2$. This means that we require a procedure to express the right-derivatives of $\psi(\cdot)$ in terms of those of $\varphi(\cdot)$.
One special case in which an explicit solution can be found is that of $J(\cdot)$ corresponding to a drifted Brownian motion. Then, $\psi(\cdot)$ is the LST of an exponential random variable with rate ${\mathfrak r}\equiv{\mathfrak r}(c,\sigma^2)\equiv -{2c}/{\sigma^2}>0$ (see, e.g., \cite[Example 3.1]{Debicki2015}). This means that the coefficients $(-1)^i\color{black}\psi_i$ are given by $i! \,{\mathfrak r}^{-i}$,  leading to
\begin{equation}\label{eq: Brownian functional}
    \mathbb{E}A_{\mathcal{P}_k}=-\frac{1}{c}\sum_{i=0}^k\frac{k!\,\mu_{k-i+1}}{(k-i)!\,(k-i+1)}{\mathfrak r}^{-i}\,.
\end{equation} 
When $J(\cdot)$ is a CPP minus a drift with rate one, Takacs' recurrence relation \cite[Theorem 5]{Takacs1962} could be applied in order to recursively compute the coefficients $\psi_i$, for $i\in\{1,\ldots,k\}$, in terms of the right-derivatives of $\varphi(\cdot)$ at $0$. The following result is a generalization of Takacs' recursion. To streamline the flow of the exposition, we first provide in Subsection~\ref{subsec: intuition} an intuitive explanation of this result and point out its inherent complications, after which we give a formal proof in Subsection~\ref{sec: proof Takacs}. 

\begin{theorem}\label{thm: Takacs}
\color{black} Let $\zeta$ be defined as in Theorem \ref{thm: basic}. \color{black} For any $i\in{\mathbb N}$, denote
\begin{equation}
    \eta_i\equiv\int_{(0,\infty)}y^ i\nu\left({\rm d}y\right).
    \end{equation}
If $\eta_k<\infty$, for $k\in{\mathbb N}$, then  \begin{equation}\label{eq: Takacs1}
            (-1)^k\mathbb{E}\zeta^k=\psi_k=-\frac{1}{(k+1)\varphi_1}\sum_{i=0}^{k-1}\binom{k+1}{i}\psi_i\varphi_{k+1-i}\,.
        \end{equation}
Otherwise, $\mathbb{E}\zeta^k=(-1)^k\psi_k=\infty$.
\end{theorem}

\begin{remark}\normalfont
    Notice that \eqref{eq: Takacs1} can be equivalently written as
\begin{align}\label{eq: Takacs expectation}
\mathbb{E}\zeta^k&=(-1)^k\psi_k=-\frac{1}{(k+1)\varphi_1}\sum_{i=0}^{k-1}\binom{k+1}{i}(-1)^i\psi_i(-1)^{k-i}\varphi_{k+1-i}\\&=\nonumber -\frac{1}{\varphi_1}\sum_{i=0}^{k-1}\binom{k}{i}\frac{(-1)^{k-i}\varphi_{k+1-i}}{k+1-i}\mathbb{E}\zeta^ i\,
\end{align}
from which it directly follows that Theorem \ref{thm: Takacs} is a generalization of \cite[Theorem 5]{Takacs1962}.
\end{remark}
Recall that
\begin{align}\label{eq: phi-eta}
    &\varphi_0=0\,,\:\:\:\:\:\varphi_1=\varrho=c+\eta_1\,,\:\:\:\:\:\varphi_2=-\left(\sigma^2+\eta_2\right) \,,\:\:\:\:\:
    \varphi_i=(-1)^{i+1}\eta_i\ \ , \ \ i\in\{3,4,\ldots\}\,.
    \end{align}
In particular, for any $i\in{\mathbb N}$, this implies that $\eta_i<\infty$ if and only if $\varphi_i<\infty$. As a result, the next corollary follows. 

\begin{corollary}\label{cor: finite condition}
    For any $k\in{\mathbb N}$, $\mathbb{E}\zeta^ k=(-1)^ k\psi_k<\infty$ if and only if $\eta_{k+1}<\infty$. 
\end{corollary}

\subsubsection{Takacs' recursion: informal derivation}\label{subsec: intuition}
Before providing the formal proof of Theorem \ref{thm: Takacs}, we first present a more intuitive treatment. The starting point is the following evident identity: by $\eqref{psi}$, for any $\alpha>0$,
\begin{equation}\label{eq: master takacs}
        \psi(\alpha)\,\varphi(\alpha)=\alpha\varphi_1\,.
    \end{equation}
    Differentiating both sides of \eqref{eq: master takacs} once leads to $\psi^{(1)}(\alpha)\,\varphi(\alpha)+\varphi^{(1)}(\alpha)\,\psi(\alpha)=\varphi_1$, while for $k\in \{2,3,\ldots\}$ we obtain, by the Leibniz rule, for any $\alpha>0$,
    \begin{equation}\label{eq: Leibniz1}
        \sum_{i=0}^{k}\binom{k}{i}\psi^{(i)}(\alpha)\varphi^{(k-i)}(\alpha)=0\,.
    \end{equation}
    As $\varphi(\cdot)$ is either increasing or decreasing in a right neighborhood of zero, it is possible to isolate $\psi^{(k)}(\alpha)$ in \eqref{eq: Leibniz1}, i.e., for every $\alpha$ in a right neighborhood of zero we have
    \begin{equation}\label{eq: psi_k}
\psi^{(k)}(\alpha)=-\frac{1}{\varphi(\alpha)}\sum_{i=0}^{k-1}\binom{k}{i}\psi^{(i)}(\alpha)\varphi^{(k-i)}(\alpha)\,.
\end{equation}
Now, $\varphi(\alpha)$ tends to zero as $\alpha\downarrow0$. Thus, if the sum in \eqref{eq: psi_k} also tends to zero, we would want to take the limit of $\psi^{(k)}(\alpha)$ as $\alpha\downarrow0$ according to l'Hopital's rule from which the recursive scheme \eqref{eq: Takacs1} follows. 

When trying to rigorize this argument, the difficulty is that is not clear how to justify  the convergence of the sum in \eqref{eq: psi_k} towards zero as $\alpha\downarrow0$. In order to resolve this issue, we rely on the following construction. For any $i\in{\mathbb N}_0$ and $m\in{\mathbb N}$, define $\Psi_{i,m}(\cdot)$ as a CPP {(with nonnegative jumps)} such that $\Psi_{i,m}(0)=\Psi_i(0)$ and  $\Delta\Psi_{i,m}(t)=m\wedge \Delta\Psi_{i}(t)$ for all $t\geqslant0$. Note that the jumps of $\Psi_{i,m}(\cdot)$ occur at the same times as in $\Psi_i(\cdot)$, and that  the jumps sizes in $\Psi_{i,m}(\cdot)$ are truncated values of the corresponding jumps sizes in $\Psi_i(\cdot)$. Then, it is possible to define a sequence (in $m$) of   spectrally positive \color{black}  L\'evy processes
\begin{equation}
    J_m(t)\equiv -ct+B_\sigma(t)+\overline{\Psi}_m(t)\ \ , \ \ t\geqslant0\,,
\end{equation}
where $\overline{\Psi}_m(\cdot)\equiv\sum_{i=1}^m\Psi_{i,m}(\cdot)$. Importantly, for each $t\geqslant0$, $J_m(\cdot)\uparrow J(\cdot)$ as $m\uparrow\infty$. Furthermore, for each $m\in{\mathbb N}$, the analogues of $\psi^{(k)}(0+)$ in the model with $J_m(\cdot)$ replacing $J(\cdot)$ are finite for every $k\in{\mathbb N}$. Thus, we first prove \eqref{eq: Takacs1} under the assumption that $J(\cdot)\equiv J_m(\cdot)$, after which we let $m\uparrow\infty$. Appealing to the appropriate convergence theorems then completes the proof for the general case. In Subsection~\ref{sec: proof Takacs} we make these steps formal. 
     
\subsubsection{Takacs' recursion: proof}\label{sec: proof Takacs}
The proof of Theorem \ref{thm: Takacs} consists of several steps,  which are described in the next lemmas.
We start by providing sufficient conditions under which \eqref{eq: Takacs1} is valid.
\begin{lemma}\label{lemma: bounded case}
    If for $k\in\mathbb{N}$ we have that $\mathbb{E}\zeta^{k+1}<\infty$, then \eqref{eq: Takacs1} is valid.  
\end{lemma}

\begin{proof}
    Differentiating both sides of \eqref{eq: master takacs} $(k+1)$ times with respect to $\alpha$ yields, according to the general Leibniz rule,
    \begin{equation}\label{eq: Leibniz}
        \sum_{i=0}^{k+1}\binom{k+1}{i}\psi^{(i)}(\alpha)\varphi^{(k+1-i)}(\alpha)=0\,,
    \end{equation}
    for every $\alpha>0$.
    As $\varphi(\cdot)$ is either increasing or decreasing in a right neighborhood of zero, we can isolate $\psi^{(k+1)}(\alpha)$ in \eqref{eq: Leibniz}, i.e., for any $\alpha$ in a right neighborhood of zero,
    \begin{equation}\label{eq: psi_k+1}
(-1)^{k+1}\psi^{(k+1)}(\alpha)=-\frac{(-1)^{k+1}}{\varphi(\alpha)}\sum_{i=0}^{k}\binom{k+1}{i}\psi^{(i)}(\alpha)\,\varphi^{(k+1-i)}(\alpha)\,.
\end{equation}
It is given that $\mathbb{E}\zeta^{k+1}<\infty$ and recall that $\varphi(\alpha)\rightarrow0$ as $\alpha\downarrow0$. Therefore, it must be that the sum in \eqref{eq: psi_k+1} vanishes as $\alpha\downarrow0$, i.e., 
\begin{equation}\label{eq: reorder}
    \sum_{i=0}^{k}\binom{k+1}{i}\psi_i\,\varphi_{k+1-i}=0\,.
\end{equation}
Finally, the result follows by isolating $\psi_k$ in \eqref{eq: reorder}.
\end{proof}
The next lemma, proved in the appendix, is needed to show that the sufficient condition stated in Lemma \ref{lemma: bounded case} is satisfied under the assumption that $J(\cdot)\equiv J_m(\cdot)$. For what follows, we need to define $\widetilde W_m(\cdot) \equiv V\wedge m+J_m(\cdot)$  with $V$ having the jump distribution of $J(\cdot)$, and in addition, for any $m\in\mathbb{N}$ we let $\widetilde{\tau}(m)$ be the first time that $\widetilde W_m(\cdot)$ hits the origin. In addition, $\psi_k(m)$ is the analogue of $\psi_k$ in the model where $\overline{\Psi}_m(\cdot)$ replaces the jump process associated with $J(\cdot)$.

\begin{lemma}\label{lemma: psi finite}
    For any $m,k\in\mathbb{N}$, $\psi_k(m)$ is finite.
\end{lemma}

The main idea behind the next two lemmas is the following. We consider \eqref{eq: Takacs1} for the model in which the jumps are bounded by $m$, i.e., we assume that the jumps are uniquely generated by the process $J_m(\cdot)$. We then show that when taking $m\to\infty$, both sides of \eqref{eq: Takacs1} actually converge to their counterparts corresponding to the jump process $J(\cdot)$ with unbounded jumps.\color{black} 

\begin{lemma}\label{lemma: psi convergence}
For any $k,m\in\mathbb{N}$, let $\psi_k(m)$ be the analogue of $\psi_k$ in the model with $W(\cdot)$ replaced by $\widetilde W_m(\cdot)$. Then, for every $k\in\mathbb{N}$,
\begin{equation}
 \psi_k(m)\rightarrow\psi_k\ \ \text{as}\ \ m\uparrow\infty\,.   
\end{equation}   
\end{lemma}

\begin{proof}
 By construction, we have that
\begin{enumerate}
    \item $\widetilde{\tau}(m)\uparrow\tau$ as $m\uparrow\infty$ , $\mathbb{P}$-a.s.;

    \item for any $t>0$ and $k\in\mathbb{N}$, 
    \begin{equation}
        0\leqslant \widetilde W^k_m(t)\,\textbf{1}_{\left[0,\widetilde{\tau}(m)\right]}(t)\uparrow W^k(t)\,\textbf{1}_{\left[0,\tau\right)}(t)\ \ \text{as}\ \ m\uparrow\infty\\ \ ,\ \ \mathbb{P}\text{-a.s.}
    \end{equation}
\end{enumerate}
As a result, by the monotone convergence theorem, 
\begin{equation}
    \mathbb{E}\widetilde{\tau}(m)\uparrow\mathbb{E}\tau\ \ \text{as}\ \ m\uparrow\infty\,,
\end{equation}
and 
\begin{align}
    \mathbb{E}\int_0^{\widetilde{\tau}(m)}\widetilde W_m^k(t)\,{\rm d}t&=\mathbb{E}\int_0^\infty \widetilde W_m^k(t)\,\textbf{1}_{\left[0,\widetilde{\tau}(m)\right]}\,{\rm d}t\\&\nonumber\uparrow \mathbb{E}\int_0^\infty W^k(t)\,\textbf{1}_{\left[0,\tau\right)}\,{\rm d}t=\mathbb{E}\int_0^{\tau}W^k(t)\,{\rm d}t\ \ \text{as}\ \ m\uparrow\infty\,.
\end{align}
 Combining the above with the definitions of $\psi_k(m)$ and $\psi_k$ yields the claim. 
\end{proof}
 
\begin{lemma}\label{lemma: phi convergence} For any $m\in\mathbb{N}$, denote by $\varphi_k(m)$ the analogue of $\varphi_k$ in the model with $W(\cdot)$ replaced by $\widetilde{W}_m(\cdot)$. If $\eta_k<\infty$, then  $\varphi_{k+1}(m)\rightarrow\varphi_{k+1}$ as $m\uparrow\infty$.
\end{lemma}
\begin{proof}
Note that for each $u\in\mathbb{N}$, $\mathcal{K}_u\equiv(-1)^{u+1}\varphi_u$ (resp.\ $\mathcal{K}_u(m)\equiv(-1)^{u+1}\varphi_u(m)$, with $m\in\mathbb{N}$) is the $u$-th cumulant of $J(1)$ (resp. $J_m(1)$, with $m\in\mathbb{N}$). In addition, 
for each $u\in\mathbb{N}$ denote  $\mathcal{J}_u\equiv\mathbb{E}\left[J(1)\right]^u$ (resp.\ $\mathcal{J}_u(m)\equiv\mathbb{E}\left[J_m(1)\right]^u$, with $m\in\mathbb{N}$). The following recursive relations are well known:
\begin{equation}\label{eq: cumulants to moments}
    \mathcal{J}_u=\sum_{a=0}^{u-1}\binom{u-1}{a}\mathcal{K}_{u-a}\mathcal{J}_a
\end{equation}
and
\begin{equation}\label{eq: moments to cumulants}
    \mathcal{K}_u=\mathcal{J}_u-\sum_{a=1}^{u-1}\binom{u-1}{a}\mathcal{K}_{u-a}\mathcal{J}_a
\end{equation}
with the initial condition $\mathcal{J}_1=\mathcal{K}_1$; see, e.g., \cite[Equations (5) and (6)]{Smith1995}.

By assumption we have $\eta_k<\infty$, and hence $\eta_i<\infty$ for any $i=1,2,\ldots,k$. Therefore, \eqref{eq: phi-eta} implies that the cumulants $\mathcal{K}_1,\mathcal{K}_2,\ldots,\mathcal{K}_k$ are all finite. Thus, a recursive application of \eqref{eq: cumulants to moments} gives us that $\mathcal{J}_1,\mathcal{J}_2,\ldots,\mathcal{J}_k$ are all finite.

Observe that the Brownian motion $B(\cdot)$ and the nonnegative jump process $\Psi(\cdot)\equiv\sum_{i=0}^\infty\Psi_i(\cdot)$ are independent, so that, for any $u\in\mathbb{N}$ we have 
\begin{equation}
    \mathcal{J}_u=\mathbb{E}\left[\sigma B(1)+c+\Psi(1)\right]^u=\sum_{a=0}^u\binom{u}{a}\mathbb{E}\left[\sigma B(1)+c\right]^a\mathbb{E}\Psi^{u-a}(1).
\end{equation}
Note that $\mathcal{J}_1,\mathcal{J}_2,\ldots,\mathcal{J}_k$ are finite, the moments of $\sigma B(1)+c$ are finite, and therefore $\mathbb{E}\Psi^u(1)<\infty$ for any $u=1,2,\ldots,k$.
Also, in a similar fashion, for any $m,u\in\mathbb{N}$ we find that 
\begin{equation}
    \mathcal{J}_u(m)=\mathbb{E}\left[\sigma B(1)+c+\overline{\Psi}_m(1)\right]^u=\sum_{a=0}^u\binom{u}{a}\mathbb{E}\left[\sigma B(1)+c\right]^a\mathbb{E}\overline{\Psi}_m^{u-a}(1).
\end{equation}
By construction, we have $0\leqslant \overline{\Psi}_m(1)\uparrow\Psi(1)$ as $m\uparrow\infty$ and hence for any $j=1,2,\ldots,k$ we have $0\leqslant \overline{\Psi}^{k-j}_{m}(1)\uparrow\Psi^{k-j}(1)$ as $m\uparrow\infty$, i.e., monotone convergence implies that $\mathbb{E}\overline{\Psi}_m^{k-j}(1)\uparrow\mathbb{E}\Psi^{k-j}(1)$ as $m\uparrow\infty$. Thus, since $\sigma B(1)+c$ has finite moments and $\mathbb{E}\Psi^u(1)<\infty$ for any $u=1,2,\ldots,k$, standard limit arithmetic as $m\uparrow\infty$ implies that $\mathcal{J}_u(m)\rightarrow J_u$ as $m\uparrow\infty$ for any $u=1,2,\ldots,k+1$ such that $\mathcal{J}_1,\mathcal{J}_2,\ldots,\mathcal{J}_k$ are finite and $\mathcal{J}_{k+1}$ is infinite iff $\mathbb{E}\Psi^{k+1}(1)=\infty$. This means that for any $u=1,2,\ldots,k$,
\begin{equation}\label{eq:limi}
    \mathcal{J}_{u+1}-\sum_{a=1}^{u}\binom{u}{a}\mathcal{K}_{u+1-a}\mathcal{J}_a=\lim_{m\uparrow\infty}\left[\mathcal{J}_{u+1}(m)-\sum_{a=1}^{u}\binom{u}{a}\mathcal{K}_{u+1-a}(m)\mathcal{J}_a(m)\right]
\end{equation}
and the limit is finite for any $u=1,2,\ldots,k-1$ and for $u=k$ it is infinite iff $\mathcal{J}_{k+1}=\infty$. Then, the claim follows recursively (by taking $u=1,2,\ldots,k$) from the following equation, obtained by combining \eqref{eq: moments to cumulants} with \eqref{eq:limi}: 
\begin{align}
    (-1)^{u+2}\varphi_{u+1}&=\mathcal{K}_{u+1}=\mathcal{J}_{u+1}-\sum_{a=1}^{u}\binom{u}{a}\mathcal{K}_{u+1-a}\mathcal{J}_a\\&\nonumber=\lim_{m\uparrow\infty}\left[\mathcal{J}_{u+1}(m)-\sum_{a=1}^{u}\binom{u}{a}\mathcal{K}_{u+1-a}(m)\mathcal{J}_a(m)\right]\\&\nonumber=\lim_{m\uparrow\infty}\mathcal{K}_{u+1}(m)=(-1)^{u+2}\lim_{m\uparrow\infty}\varphi_{u+1}(m)\,,
\end{align}
where for $u=0$ we have
\begin{equation}
\varphi_1=\mathcal{K}_1=\mathcal{J}_1=\lim_{m\uparrow\infty}\mathcal{J}_1(m)=\lim_{m\uparrow\infty}\mathcal{K}_1(m)=\lim_{m\uparrow\infty}\varphi_1(m)\,.
\end{equation}
This concludes the proof.
\end{proof}
 

By applying the previous lemmas, we conclude the validity of the first part of Theorem \ref{thm: Takacs}.
\begin{lemma}\label{lemma: bounded recursion}
   If $\eta_k<\infty$, then \eqref{eq: Takacs1} is valid.
\end{lemma}

\begin{proof}
    For each $j,m\in\mathbb{N}$, Lemma \ref{lemma: bounded case} and Lemma \ref{lemma: psi finite} imply that
    \begin{equation}\label{eq: psi_k(m)}
        \psi_j(m)=-\frac{1}{(j+1)\varphi_1(m)}\sum_{i=0}^{j-1}\binom{j+1}{i}\psi_i(m)\varphi_{j+1-i}(m)\,.
    \end{equation}
    In addition, $\eta_k<\infty$ implies that $\eta_i<\infty$ for any $i\in\{1,\ldots,k\}$. Therefore, by taking $m\uparrow\infty$ on both sides of \eqref{eq: psi_k(m)} (for any $i\in\{1,\ldots,k\}$), appealing to Lemmas \ref{lemma: psi convergence} and  \ref{lemma: phi convergence} implies the required result.
\end{proof}
The second part of Theorem \ref{thm: Takacs} follows from the next lemma.
\begin{lemma}
    If $\eta_k=\infty$, then  $\mathbb{E}\zeta^k=(-1)^k\psi_k=\infty$.
\end{lemma}
\begin{proof}
Recall \eqref{eq: phi-eta} which states that $\varrho=c+\eta_1$, i.e., as $\varrho$ is finite so is $\eta_1$.
 In addition, $\eta_k=\infty$ is given and therefore there is some $k_0\in\{1,\ldots,k-1\}$ such that $\eta_{k_0}<\infty$ and $\eta_{k_0+1}=\infty$. Lemma \ref{lemma: bounded recursion} implies that 
 \begin{equation}\label{eq: k_0}
    (-1)^{k_0}\mathbb{E}\zeta^{k_0}=\psi_{k_0}=-\frac{1}{(k_0+1)\varphi_1}\sum_{i=0}^{k_0-1}\binom{k_0+1}{i}\psi_i\varphi_{k_0+1-i}\,.
 \end{equation}
Note that $\zeta$ is a nonnegative random variable whose distribution is not fully concentrated at zero and therefore its moments are positive, i.e., $\psi_u\neq0$ for any $u\in\mathbb{N}$. In addition,  $\eta_{k_0+1}=\infty$ and hence \eqref{eq: phi-eta} implies that $\varphi_{k_0+1}$ is infinite. Consequently, the right-hand side of \eqref{eq: k_0} is infinite which means that $\mathbb{E}\zeta^{k_0}=\infty$. Therefore, as $k>k_0$, it follows that $\mathbb{E}\zeta^k=\infty$.    
\end{proof}

\section{High-order moments}\label{sec: high-order moments}
Where the previous section considered the first moment of the vector $\left(A_{f_1},A_{f_2},D_{g_1},D_{g_2}\right)$, we now extend the analysis to high-order moments. 
The main result is Theorem \ref{thm: Gamma}, providing a useful representation. Again, it leads to an explicit computation scheme in the case of polynomial $f(\cdot)$.

We start by introducing some convenient notation. Let $f(\cdot)$ be a nonnegative Borel function on $\mathbb{R}_+$ and denote
\begin{equation}
    \Gamma_af\left(x\right)\equiv \mathbb{E}A_f(x)\ \ , \ \ \Gamma_df\left(x\right)\equiv \mathbb{E}D_f(x)\,,
\end{equation}
where $A_f(x)$ and $D_g(x)$ are defined in $\eqref{eq: functional def}$. In particular, note that Theorem \ref{thm: D_f} implies that $\Gamma_af(x)=\lambda\, \Gamma_df(x)$.
Let $g(\cdot)$ be a nonnegative Borel function on $\mathbb{R}_+$, and for $i\in\{a,d\}$ define a new function $g\Gamma_if(\cdot)$ via
\begin{equation}
    g\Gamma_if(x)\equiv g(x)\cdot\Gamma_if(x)\ \ , \ \ \forall x\geqslant 0\,.
\end{equation}
Furthermore, observe that for $i\in\{a,d\}$ and a constant $\epsilon>0$, the linearity of expectation implies that 
\begin{equation}
    g\Gamma_i(\epsilon f)\equiv \epsilon g\Gamma_if\,.
\end{equation}
As a result,  due to Theorem \ref{thm: D_f},
\begin{equation}
    g\Gamma_d(\epsilon f)\equiv\lambda g\Gamma_a(\epsilon f)=\epsilon\lambda\Gamma_af\,.
\end{equation}
for any $\epsilon>0$. A recursive application of this identity leads to  Lemma \ref{lemma: recursion} below. Notice that for $i=1,2$, $j_i\in\{a,d\}$ and Borel functions $f_i:[0,\infty)\rightarrow[0,\infty)$ we apply the notation $\Gamma_{j_1}f_1\Gamma_{j_2}f_2\equiv\Gamma_{j_1}\left(f_1\Gamma_{j_2}f_2\right)$. Importantly, note that there is a difference between $\Gamma_{j_1}f_1\Gamma_{j_2}f_2$ and $\left(\Gamma_{j_1}f_1\right)\left(\Gamma_{j_2}f_2\right)$. 
\begin{lemma}\label{lemma: recursion}
Let $k\in\{2,3,\ldots\}$. Assume that $f_1,f_2,\ldots,f_k$ are nonnegative  Borel functions on $\mathbb{R}_+$. Then, for any  $\textbf{j}\equiv(j_1,j_2,\ldots,j_k)\in\left\{a,d\right\}^k$,    
\begin{equation}
\Gamma_{j_1}f_1\Gamma_{j_2}f_2\Gamma_{j_3}f_3\ldots \Gamma_{j_k}f_k\equiv\lambda^{u_{\textbf{j}}}f_1\Gamma_a f_2\Gamma_a\ldots f_{k-1}\Gamma_af_k\,,    
\end{equation}
where $u_{\textbf{j}}\equiv\left|\,m\in\{1,\ldots, k\}:\,j_m=d\,\right|$.
\end{lemma}
Lemma \ref{lemma: recursion} is crucial in order to prove the following theorem.
In the sequel we denote the set of all permutations of $\{1,2,3,\ldots,k\}$ by $\pi[1:k]$.

\begin{theorem}\label{thm: Gamma}
    Let $\ell,k\in\mathbb{N}_0$ such that $k+\ell\in{\mathbb N}$. In addition, assume that $f_1,f_2,\ldots,f_{k+\ell}$ are nonnegative Borel functions on $\mathbb{R}_+$. Then, 
    \begin{equation}\label{eq: product}
    \mathbb{E}\Big[\prod_{i=1}^k\prod_{j=1}^\ell A_{f_i}D_{f_{k+j}}\Big]=\ \lambda^\ell\sum_{\pi\in\pi[1:k+\ell]}\int_0^\infty\Gamma_af_{\pi(1)}\Gamma_af_{\pi(2)}\Gamma_af_{\pi(3)}\Gamma_a\ldots\Gamma_af_{\pi(k+\ell)}(x)\,{\rm d}G_V(x) \,.
    \end{equation}
\end{theorem}
\begin{proof} It is enough to provide the proof under the assumption that $V=x$, $\mathbb{P}$-a.s., i.e.,  
\begin{equation}\label{eq: condition x}
    G_V(y)\equiv  \begin{cases} 
      1 & y\geqslant x \\
      0 & y<x 
   \end{cases}\ \ , \ \ \forall y\in\mathbb{R}
\end{equation}
for some constant $x\in(0,\infty)$, as the more general case will follow via conditioning and un-conditioning on $V$. For brevity, we denote the expectation under the assumption \eqref{eq: condition x} by $\mathbb{E}_x$. In addition, for $i\in\{1,\ldots,k+\ell\}$, we denote, for $ t\geqslant0$,
\begin{equation}
    F_i(t)\equiv \begin{cases} 
      t, & i\in\{1,\ldots, k\}\,, \\
      N(t), &i\in\{k+1,\ldots, k+\ell\} \,,
   \end{cases}\ \ 
\end{equation}
where $N(\cdot)$ is defined in \eqref{eq: N def} as the arrival process of the jumps of the process $W(\cdot)$. We also introduce, for $i\in\{1,\ldots,k+\ell\}$,
\begin{equation}
    e_i\equiv \begin{cases} 
      a, & i\in\{1,\ldots, k\}\,, \\
      d,&i\in\{k+1,\ldots, k+\ell\} \,.
   \end{cases}\,
\end{equation} 
Now, observe that
 \begin{equation}
\mathbb{E}_x\Big[\prod_{i=1}^k\prod_{j=1}^\ell A_{f_i}D_{f_{k+j}}\Big]=\sum_{\pi\in\pi[1:k+\ell]} E_\pi\,,
 \end{equation}
 where 
 \begin{equation}
E_\pi\equiv\mathbb{E}_x\int_0^\tau\int_{t_1}^{\tau}\int_{t_2}^\tau\ldots\int_{t_{k+\ell-1}}^\tau\prod_{i=1}^{k+\ell}f_{\pi(i)}\circ W(t_i)\bigotimes_{i=1}^{k+\ell}{\rm d}F_{\pi(i)}(t_i)
 \end{equation}
 for each $\pi\in\pi[1:k+\ell]$. Consider some $\pi\in\pi[1:k+\ell]$
and since $W(\cdot)$ is a Markov process, conditioning on $W(t_{k+\ell-1})$ reveals that $E_\pi$ is equal to
\begin{align}
\mathbb{E}_x\int_0^\tau\int_{t_1}^{\tau}\int_{t_2}^\tau\ldots\int_{t_{k+\ell-2}}^\tau f_{\pi(k+\ell-1)}\Gamma_{e_{\pi(k+\ell)}}f_{\pi(k+\ell)}\circ W(t_{k+\ell-1})\prod_{i=1}^{k+\ell-2}f_{\pi(i)}\circ W(t_i)\bigotimes_{i=1}^{k+\ell-1}{\rm d}F_{\pi(i)}(t_i)\,.
 \end{align}
 Therefore, by conditioning on $W(t_{k+\ell-2})$, we find that $E_\pi$ is equal to
 \begin{align}
\nonumber\mathbb{E}_x\int_0^\tau\int_{t_1}^{\tau}\int_{t_2}^\tau\ldots\int_{t_{k+\ell-3}}^\tau f_{\pi(k+\ell-2)}\Gamma_{e_{\pi(k+\ell-1)}}f_{\pi(k+\ell-1)}&\Gamma_{e_{\pi(k+\ell)}}f_{\pi(k+\ell)}\circ W(t_{k+\ell-2})\\&\prod_{i=1}^{k+\ell-3}f_{\pi(i)}\circ W(t_i)\bigotimes_{i=1}^{k+\ell-2}{\rm d}F_{\pi(i)}(t_i)\,.
\end{align}
Along the same lines, we can condition sequentially on $W(t_{k+l-3}),W(t_{k+l-4}),\ldots, W(t_2)$, so as to eventually obtain
\begin{align}
    E_\pi&={\mathbb E}_x\int_0^\tau f_{\pi(2)}\Gamma_{e_{\pi(2)}}f_{\pi(3)}\Gamma_{e_{\pi(3)}}\ldots\Gamma_{e_{\pi(k+\ell)}}f_{\pi(k+\ell)}\circ W(t-)f_{\pi(1)}\circ W(t)\,{\rm d}F_{\pi(1)}(t)\\&\nonumber=\Gamma_{e_{\pi(1)}}f_{\pi(1)}f_{\pi(2)}\Gamma_{e_{\pi(2)}}f_{\pi(3)}\Gamma_{e_{\pi(3)}}\ldots\Gamma_{e_{\pi(k+\ell)}}f_{\pi(k+\ell)}(x)\\&\nonumber=\lambda^\ell\Gamma_af_{\pi(1)}f_{\pi(2)}\Gamma_a f_{\pi(3)}\Gamma_a\ldots\Gamma_a f_{\pi(k+\ell)}(x)\,,
\end{align}
where the last equality follows by applying Lemma \ref{lemma: recursion}. 
\end{proof}

\begin{remark}\normalfont
In particular, if $f_i(\cdot)\equiv f(\cdot)$ for all $i=1,\ldots,k+\ell$, then \eqref{eq: product} is equal to 
    \begin{equation}
        \mathbb{E}\big[A^k_f\,D_f^\ell\big]=\ \lambda^\ell(k+\ell)!\int_0^\infty\Gamma_af\Gamma_af\Gamma_af\Gamma_a\ldots\Gamma_af(x)\,{\rm d}G_V(x) \,, \end{equation}
where the integrand on the right-hand side is a $(k+\ell)$-fold composition. 
\end{remark}
\begin{remark}\normalfont
    Note that in practice, to apply Theorem \ref{thm: Gamma}, one must know how to compute $\Gamma_if(x)$ for every $i\in\{a,d\}$ and $x$. The tools needed for these computations have been provided in Section \ref{sec: first moment}. A demonstration of these tools, in the context of Theorem \ref{thm: Gamma}, is provided in Section \ref{sec: demonstration}. 
\end{remark}

As before, when specializing to polynomial $f(\cdot)$, the objects of interest simplify drastically.
In order to compute $\Gamma_a\mathcal{P}_{k_m}\Gamma_a\mathcal{P}_{k_{m-1}}\ldots\Gamma_a\mathcal{P}_{k_1}(x)$,  Corollary \ref{cor: polynomial} can be applied recursively. Indeed,  Corollary~\ref{cor: polynomial} implies that
\begin{equation}\label{eq: two iterations}
    \mathcal{P}_{k_2}\Gamma_a\mathcal{P}_{k_1}(x)=-\frac{1}{\varrho}\sum_{i_1=0}^{k_1}\binom{k_1}{i_1}\frac{(-1)^{i_1}\psi_{i_1}}{k_1-i_1+1}x^{k_1+k_2+1-i_1}
\end{equation}
and hence an additional application of Corollary \ref{cor: polynomial} yields that $\Gamma_a\mathcal{P}_{k_2}\Gamma_a\mathcal{P}_{k_1}(x)$ equals
\begin{equation}\label{eq: polynomial22}
\frac{1}{\varrho^2}\sum_{i_1=0}^{k_1}\binom{k_1}{i_1}\frac{(-1)^{i_1}\psi_{i_1}}{k_1-i_1+1}\sum_{i_2=0}^{k_1+k_2+1-i_1}\binom{k_1+k_2+1-i_1}{i_2}\frac{(-1)^{i_2}\psi_{i_2}}{k_1+k_2-i_1-i_2+2}x^{k_1+k_2-i_1-i_2+2}\,.
\end{equation}

Iterating this computation gives the following more general result.
\begin{corollary}\label{cor: polynomial2}
Let $m\in{\mathbb N}$ and consider $k_1,k_2,\ldots,k_m\in\mathbb{N}_0$. Then, for  $x>0$,
\begin{align}\Gamma_a\mathcal{P}_{k_m}\Gamma_a\mathcal{P}_{k_{m-1}}\ldots\Gamma_a\mathcal{P}_{k_1}(x)\end{align} is equal to 
\begin{align}
    (-\varrho)^{-m}\sum_{i_1=0}^{k_1}&\binom{k_1}{i_1}\frac{(-1)^{i_1}\psi_{i_1}}{k_1-i_1+1}\sum_{i_2=0}^{k_1+k_2-i_1+1}\binom{k_1+k_2-i_1+1}{i_2}\frac{(-1)^{i_2}\psi_{i_2}}{k_1+k_2-i_1-i_2+2}\ldots\\&\nonumber\ldots\sum_{i_m=0}^{k_m+\sum_{j=1}^{m-1}(k_j+1-i_j)}\binom{k_m+\sum_{j=1}^{m-1}(k_j+1-i_j)}{i_m}\frac{(-1)^{i_m}\psi_{i_m}}{\sum_{j=1}^m(k_j+1-i_j)}x^{\sum_{j=1}^m(k_j+1-i_j)}\,.
\end{align}
\end{corollary}

\begin{remark}\normalfont \label{R7}
    Importantly, Theorem \ref{thm: Gamma} and Corollary \ref{cor: polynomial2} are also useful in order to compute \textit{joint} moments of polynomial functionals. For example, in Section \ref{sec: demonstration}, it is demonstrated how to apply them (with $f_1\equiv1$ and $f_2={\rm Id}$) in order to compute the joint moments of $\tau$ and $A\equiv\int_0^\tau W(t)\,{\rm d}t$ when $J(\cdot)$ is a CPP with nonnegative jumps and $V$ has the jump distribution of $J(\cdot)$.  
\end{remark}
\section{The process $x\mapsto A_f(x)$}\label{sec: cost process}
{The functional $A_f(x)$ is determined by the choice of $f(\cdot)$ together with the value of the initial condition $x>0$. In this section, we fix $f(\cdot)$ and analyze the sensitivity of $A_f(x)$ with respect to $x$. Put differently, for given $f$ we study the properties of the stochastic process $x\mapsto A_f(x)$. A specific motivation for this sensitivity analysis lies in a storage system in which:}
\begin{enumerate}
   \item $W(t)$ is the amount of stored work at time $0\leqslant t\leqslant\ \tau$,
    \item $f(w)$ and is the holding cost of $w\geqslant 0$ workload units per an infinitesimal unit of time.
\end{enumerate}
\color{black} In addition, let this storage system have a controller who knows at time $t=0$ that the initial workload has the value $x_1\geqslant 0$. Suppose that this controller needs to decide whether or not to accept a new `external' job of size $x_2\geqslant 0$ at time $t=0$. We assume that it is impossible to introduce new external jobs during a busy period and hence the decision should be based on the distribution of the difference $A_f(x_1+x_2)-A_f(x_1)$ representing the  expected additional holding cost which is due to accepting the new job. \color{black} For a few examples of papers on similar functionals, playing a pivotal role in the context of solving various types of social optimization problems in service systems, we refer to e.g., \cite{Haviv1998,Haviv2018a,Haviv2018b,Jacobovic2022a}.
\color{black}

This example illustrates the relevance of studying the process $x\mapsto A_f(x)$. Note that this process has a clear economic interpretation: for each $x\geqslant 0$, $A_f(x)$ represents the total holding cost accumulated during the time interval $[0,\tau_x]$, i.e., during the first busy period. This interpretation explains that in the sequel we refer to it as the \textit{cost process}. For other examples of related cost processes associated with the workload process in an M/G/1 queue, along with their motivations and analysis, we refer to \cite{Jacobovic2024,Jacobovic2023, Jacobovic2022}. For a recent work on such cost processes in a more general framework, along with additional application areas, see \cite{Bodas2024}.
\color{black}

The objective of this section is to apply results from the previous sections to prove that for any polynomial $f(\cdot)$ the corresponding cost process has sample paths which are $\mathbb{P}$-a.s.\ Lipschitz continuous on $\mathbb{R}_+$. We also provide a method to compute the high-order moments of the finite-dimensional distributions of the cost process. Lastly, we discuss an integro-differential equation {which is related to the infinitesimal generator of $x\mapsto A_f(x)$ when $f(\cdot)={\rm Id}$, giving rise to alternative ways to evaluate the corresponding moments.}

\medskip

\color{black}
In the first step of our analysis of $A_f(\cdot)$ we derive the following decomposition of $A_f(x+y)$ (for some $y>0$):
\begin{align}\label{eq: cost decomposition}
   \color{black} A_f(x+y)=\int_0^{\tau_x}f\left[W_x(t)+y\right]{\rm d}t+\int_{0}^{\tau_{x+y}-\tau_x}f\left[x+y+J\left(\tau_x+t\right)\right]{\rm d}t\,.
\end{align}
\color{black} The validity of this decomposition can be seen via the following line of reasoning. First notice that $[0,\tau_{x+y}]=[0,\tau_x]\cup[\tau_x,\tau_{x+y}]$ and by definition $W_{x+y}(t)=y+W_x(t)$ for every $t\in[0,\tau_x]$. Therefore, deduce that
\begin{equation}
 \int_0^{\tau_x}f\left[W_{x+y}(t)\right]{\rm d}t=\int_0^{\tau_x}f\left[W_x(t)+y\right]{\rm d}t.   
\end{equation}  
In addition, 
$\tau_{x+y}-\tau_x\geqslant y$ and $W_{x+y}(\tau_x)=y$. Thus, as $W_{x+y}(t)=x+y+J(t)$ for every $t\in[0,\tau_{x+y}]$ and $J(\tau_x)=-x$, then a standard change in the integration variable implies that
\begin{equation}
    \int_{\tau_x}^{\tau_{x+y}}f\circ W_{x+y}(t){\rm d}t=\int_0^{\tau_{x+y}-\tau_x}f\left[x+y+J(\tau_x+t)\right]{\rm d}t,
\end{equation}
thus concluding \eqref{eq: cost decomposition}. 

\color{black}
Importantly, due to the strong Markov property of $J(\cdot)$ and the fact that $W_{x+y}(\tau_x)=y$ \color{black} (i.e., $x+J(\tau_x)=0$)\color{black}, the first term and the second term are independent entailing that the second term is distributed as $A_f(y)$. In the sequel, we restrict our attention to the case where $f(\cdot)$ is a polynomial; for simplicity of notation, for any $k\in{\mathbb Z}_+$ we denote (with mild abuse of notation) $A_k(\cdot)\equiv A_{\mathcal{P}_k}(\cdot)$. For this $f(\cdot)$, by the binomium, the first term in \eqref{eq: cost decomposition} simplifies to
\begin{equation}
    \sum_{i=0}^{k\color{black}}\binom{k}{i}y^{k-i}A_i(x)\ \ , \ \ x\geqslant0\,.
\end{equation}
As it turns out, this representation has several applications which are presented in the next subsections.  

\begin{remark}
    \normalfont In our setup $A_k(\cdot)$ is a cost process which is based on the cost functional $A_{\mathcal{P}_k}(\cdot)$. As is readily checked, all of the analysis to follow for $A_k(\cdot)$ can be translated to an analogue cost process $D_k(\cdot)$ which is based on the cost functional $D_{\mathcal{P}_k}(\cdot)$. 
\end{remark}

\subsection{Auto-covariance function}
If $\eta_{2k+2}<\infty$, then Corollary \ref{cor: finite condition}, in combination with \eqref{eq: polynomial22}, implies that the auto-covariance function of the process $A_k(\cdot)$ is well defined. It is given by 
\begin{equation}\label{eq: COV}
\text{Cov}\left[A_k(x_1),A_k(x_2)\right]=\mathbb{E}[A_k(x_1)A_k(x_2)]-\mathbb{E}A_k(x_1)\,\mathbb{E}A_k(x_2)\ \ , \ \ 0\leqslant x_1\leqslant x_2
\end{equation}
where the second term can be computed by applying Corollary \ref{cor: polynomial}. Now, it is to be explained how to utilize the previous results in order to compute the joint moment appearing in \eqref{eq: COV}. 

To this end, we introduce a random variable $\widetilde{A}_k(x_2-x_1)$ which is distributed as $A_k(x_2-x_1)$ and which is independent from $A_k(x_1)$. Then, the decomposition \eqref{eq: cost decomposition} implies that
\begin{align}
    \label{corr}\mathbb{E}[A_k(x_1)A_k(x_2)]&=\mathbb{E}\left[A_k(x_1)\Big(\sum_{i=0}^k\binom{k}{i}\color{black}(x_2-x_1)\color{black}^{k-i}A_i(x_1)+\widetilde{A}_k(x_2-x_1)\Big)\right]\\&\nonumber=\sum_{i=0}^{k}\binom{k}{i}\color{black}(x_2-x_1)\color{black}^{k-i}\mathbb{E}[A_k(x_1)A_i(x_1)]+\mathbb{E}A_k(x_2-x_1)\,\mathbb{E}A_k(x_1)\,.
\end{align}
Note that all of the expectations in the right-hand side of \eqref{corr} can be computed by the method which was developed in the previous sections. More generally, such an approach is useful in order to compute all the joint moments of the finite-dimensional distributions of $A_k(\cdot)$, i.e., the high-order moments of the coordinates of the random vector
\begin{equation}
    \big(A_k(x_1),A_k(x_2),\ldots,A_k(x_m)\big)
\end{equation}
for $0\leqslant x_1<x_2<\ldots<x_m<\infty$ and $m=2,3,\ldots$. 

\subsection{Lipschitz continuity}
The next theorem states that the cost functional is \color{black}locally \color{black} Lipschitz continuous with respect to the initial workload $x\geqslant 0$. 
\begin{theorem}\label{thm: Lipschitz}
Let $k\in\mathbb{N}_+$ and assume that $\eta_{k+1}<\infty$. Then, $x\mapsto A_k(x)$ is $\mathbb{P}$-a.s.\ \color{black} locally \color{black} Lipschitz 
 continuous on $\mathbb{R}_+$.
\end{theorem}

\begin{proof}
   Let $x$ be an arbitrary positive number. Also,  notice that, for any $0\leqslant x_0\leqslant x$,
   \begin{align}\label{eq: derivative upper bound}
       \color{black}\limsup_{y\downarrow0}\frac{A_k(x_0+y)-A_k(x_0)}{y}&\color{black}=kA_{k-1}(x_0)+\limsup_{y\downarrow0}\frac{1}{y}\int_{0}^{\tau_{x_0+y}-\tau_{x_0}}\left[\color{black}x_0+\color{black}y+J\left(\tau_{x_0}+t\right)\right]^k {\rm d}t\,.
   \end{align} In order to show that the $\limsup$ in right-hand side of \eqref{eq: derivative upper bound} is $\mathbb{P}$-a.s. finite, consider the event
   \begin{equation}
       \mathcal{E}(x_0)\equiv \left\{\omega\ :\ \limsup_{y\downarrow0}\frac{1}{y}\int_{0}^{\tau_{x_0+y}-\tau_{x_0}}\left[\color{black}x_0+\color{black}y+J\left(\tau_{x_0}+t\right)\right]^k {\rm d}t=\infty\right\}
   \end{equation}
   and its complement by $\mathcal{E}^c(x_0)$. Notably, for every $\omega\in\mathcal{E}(x_0)$, there is a sequence $y_n\equiv y_n(\omega,x_0)$ with  $n\in{\mathbb N}$ such that $y_n\downarrow0$ as $n\to\infty$ and
   \begin{equation}
       0\leqslant\frac{1}{y_n}\int_{0}^{\tau_{x_0+y_n}-\tau_{x_0}}\left[\color{black}x_0+\color{black}y_n+J\left(\tau_{x_0}+t\right)\right]^k {\rm d}t\,\uparrow\,\infty\ \ \text{as} \ \ n\to\infty\,.
   \end{equation}
Therefore, if $\mathbb{P}\left\{\mathcal{E}(x_0)\right\}>0$, the monotone convergence theorem yields that
\begin{align}\label{eq: contradiction} \mathbb{E}\Big[\frac{1}{y_n}&\int_{0}^{\tau_{x_0+y_n}-\tau_{x_0}}\left[\color{black}x_0+\color{black}y_n+J\left(\tau_{x_0}+t\right)\right]^k {\rm d}t\Big]\\&\geqslant\mathbb{E}\Big[\frac{\textbf{1}_{\mathcal{E}(x_0)}}{y_n}\int_{0}^{\tau_{x_0+y_n}-\tau_{x_0}}\left[\color{black}x_0+\color{black}y_n+J\left(\tau_{x_0}+t\right)\right]^k {\rm d}t\Big]\rightarrow\infty\ \ \text{as}\ \ n\to\infty\,.  \nonumber \end{align}
On the other hand, the assumption $\eta_{k+1}<\infty$, in combination with  Corollaries \ref{cor: polynomial} and  \ref{cor: finite condition}, imply that, for any $x, y>0$,
\begin{align}
    \mathbb{E}\Big[\frac{1}{y}\int_{0}^{\tau_{x+y}-\tau_x}&\left[\color{black}x+y +J\left(\tau_x+t\right)\right]^k {\rm d}t\Big]=\frac{\mathbb{E}A_k(y)}{y}\\&=\nonumber-\frac{1}{\varrho}\sum_{i=0}^k\binom{k}{i}\frac{y^{k-i}}{k-i+1}(-1)^i\psi_i\xrightarrow{y\downarrow0}\frac{(-1)^{k+1}\psi_k}{\varrho}<\infty\,.
\end{align}
As a result, a contradiction to \eqref{eq: contradiction} is obtained, and hence $\mathbb{P}\{\mathcal{E}(x_0)\}=0$. Therefore, there exists (random) $K_{x_0},\delta_{x_0}\in(0,\infty)$ such that $\mathbb{P}\text{-a.s.}$,
\begin{equation}
    |A_k(z_1)-A_k(z_2)|<K_{x_0}|z_2-z_1|
\end{equation}
for every $z_1$ and $z_2$ which belong to a set $B(x_0)\equiv B_{\delta_{x_0}}(x_0)\equiv(x_0-\delta_{x_0},x_0+\delta_{x_0})$. Now, notice that $\cup_{x_0\in\mathbb{Q}\cap[0,x]}B(x_0)$ is an open cover of $[0,x]$. Therefore, the Heine-Borel theorem implies that there are $m\in\mathbb{N}$ and $q_1,q_2,\ldots,q_m$ in $\mathbb{Q}\cap[0,x]$ such that  $\cup_{i=1}^mB(q_i)=[0,x]$. In addition, let $\beta$ be the Lebesgue number of the open cover $\cup_{i=1}^mB(q_i)$. Now, for every $0\leqslant z_1<z_2\leqslant x$ there are numbers $s\in\{2,3,\ldots\}$ and $z_1=u_1\leqslant u_2\leqslant u_3\leqslant\ldots\leqslant u_{s}=z_2$ such that $u_{j+1}-u_j\leqslant\beta$ for every $j\in\{1, \ldots, s-1\}$. As a result, with $K\equiv\max\left\{K(q_i):i\in\{1,\ldots, m\}\right\}$, on the event $\mathcal{E}\equiv\cap_{i=1}^m\mathcal{E}^c(q_i)$ we have 
\begin{align}
    A_k(z_2)-A_k(z_1)&\leqslant\sum_{j=1}^{s-1}\left[A_k(u_{j+1})-A_k(u_{j})\right]\leqslant K\sum_{j=1}^ {s-1}\left(u_{j+1}-u_j\right)\leqslant K(z_2-z_1)\,.
\end{align}
Therefore, as $\mathbb{P}\{\mathcal{E}\}=1$, we know that $A_k(\cdot)$ is Lipschitz continuous on $[0,x]$, $\mathbb{P}$-a.s. Finally, the result follows since $\mathbb{R}_+=\cup_{x=1}^\infty[0,x]$.   
\end{proof}
Since Lipschitz continuity implies differentiability a.e., the following corollary is an immediate consequence of Theorem \ref{thm: Lipschitz}.
    \begin{corollary}
    $\mathbb{P}$-a.s. the limit    \begin{equation}
    L(x)\equiv\lim_{y\downarrow0}\frac{1}{y}\int_{0}^{\tau_{x+y}-\tau_{x}}\left[\color{black}x+\color{black}y+J\left(\tau_{x}+t\right)\right]^k {\rm d}t    \end{equation}
    exists a.e.\ on $\mathbb{R}_+$. Furthermore, wherever this limit exists, $A_k(\cdot)$ is $\mathbb{P}$-a.s.\ differentiable with the derivative satisfying
    \begin{equation}
        A'_k(x)=kA_{k-1}(x)+L(x)\,.
    \end{equation}
    \end{corollary}

\begin{remark}
    \normalfont 
    Note that $A_{k-1}(\cdot)$ is nondecreasing, and that the distribution of $L(x)$ does not depend on $x$. Therefore, it is tempting to conclude that the sample paths of $A_k(\cdot)$ are $\mathbb{P}$-a.s.\ convex. This conclusion is, however, incorrect since it is possible to have a sequence of stochastically increasing random variables that with positive probability will not have increasing values. For an example of a slightly different cost process in an M/G/1 setup which has convex sample paths, we refer to \cite[Corollary 2]{Jacobovic2022}. More examples of convex stochastic processes which arise in various  applied probability models are found in \cite{Jacobovic2020}.
\end{remark}

\subsection{Differential equation}
In this subsection we discuss a specific (ordinary) differential equation by which we can characterize moments of $A_f(x)$ in the case $f(\cdot)={\rm Id}$. For simplicity, we introduce it for the special case in which $k=1$ and under the assumption that $\eta_i<\infty$ for all $i\in{\mathbb N}$. We conclude this section by pointing out possible interpretations which may
lead to possible directions for future research.

For $u\in\mathbb{N}$ we introduce the set
\begin{equation}
    \mathcal{I}(u)\equiv\left\{{\boldsymbol i}\in\mathbb{N}^u:  \ i_j\leqslant 2j-1+\sum_{a=1}^{u-1}i_a\  ,\ j\in\{ 1,\ldots, u\} \ \text{and} \ \sum_{j=1}^ui_j=2u-1\right\};
\end{equation}
which is not empty because $(0,0,\ldots,0,2u-1)\in\mathcal{I}(u)$. Also, define
\begin{equation}
    \theta(u)\equiv u!(-\rho)^{-u}\sum_{\boldsymbol{i}\in\mathcal{I}(u)}\prod_{j=1}^u\binom{2j-1-\sum_{a=1}^{j-1}i_a}{i_j}\frac{(-1)^{i_j}\psi_{i_j}}{2j-\sum_{a=1}^ji_a}.
\end{equation}
\begin{theorem}
    Assume that $\eta_i<\infty$ for all $i\in{\mathbb N}$. 
Then, the following recursion in $\ell$ applies:
\begin{equation}\label{eq: recusrion ODE}
    \mathbb{E}A_1^\ell(x)= \sum_{i=0}^{\ell-1}\binom{\ell}{i}\theta(\ell-i)\,\int_0^x\mathbb{E}A^i_1(y)\,{\rm d}y\ \ , \ \ x\geqslant0\,.
\end{equation}
\end{theorem}
\begin{proof}
    Applying Theorem \ref{thm: Gamma} and Corollary \ref{cor: polynomial2}, one can  verify that, for any such $u$,
\begin{equation}
    \theta(u)=\lim_{y\downarrow0}\frac{\mathbb{E}A_1^u(y)}{y}\,.
\end{equation}
Now, for an integer $\ell$ the decomposition \eqref{eq: cost decomposition}  yields that, for any $x,y>0$,
\begin{equation}
    \mathbb{E}A^\ell_1(x+y)=\mathbb{E}\left[A_1(x)+y\tau_x+\widetilde{A}_1(y)\right]^\ell\,.
\end{equation}
As a consequence,
\begin{align}\label{eq: pre-derivative}
    \frac{\mathbb{E}A^\ell_1(x+y)-\mathbb{E}A^\ell_1(x)}{y}&=\frac{1}{y}\sum_{i=0}^{\ell-1}\binom{\ell}{i}\mathbb{E}\left[A_1^i(x)\left[y\tau_x+\widetilde{A}_1(y)\right]^{\ell-i}\right]\\&=\sum_{i=0}^{\ell-1}\sum_{j=0}^{\ell-i}\binom{\ell}{i}\binom{\ell-i}{j}y^j\mathbb{E}\big[A^i_1(x)\overline{\tau}^j_x\big]\,\frac{\mathbb{E}\widetilde{A}_1^{\ell-i-j}(y)}{y}\,.\nonumber
\end{align}
Thus, by taking the limit as $y\downarrow0$ from both sides of \eqref{eq: pre-derivative}, we thus obtain the  differential equation
\begin{equation}\label{eq: ode}
    \frac{{\rm d}}{{\rm d}x}\mathbb{E}A_1^\ell(x)=\sum_{i=0}^{\ell-1}\binom{\ell}{i}\,\theta(\ell-i)\,\mathbb{E}A^i_1(x)\,.
\end{equation}
An integration of both sides, using the initial condition $A_1^ \ell(0)=0$, yields the recursion.
\end{proof}

\color{black}
\begin{remark}
    \normalfont Equation
\eqref{eq: ode} is particularly intriguing as it provides an expression for  the infinitesimal generator of the process $A^\ell_1(\cdot)$ in terms of the processes $A_1^ i(\cdot)$ for $i\in\{1,\ldots, \ell-1\}$. This argument may be repeated for each of the processes $A_1^ i(\cdot)$ for $i\in\{1,\ldots, \ell-1\}$, indicating that there is an `integro-differential relation' between the underlying infinitesimal generators. Backed by the findings of \cite{Glynn1994}, one may anticipate that
there is a relation to a corresponding Poisson's equation. In this respect, it is noticed that \eqref{eq: cost decomposition} shows that $x\mapsto \left(A_1(x),\tau_x\right)$ is a Markov process, even though $A_1(\cdot)$ is {\it not} a Markov process on its own. 

\end{remark}
\color{black}

\begin{remark}
    \normalfont 
    The `integro-difference equation' \eqref{eq: recusrion ODE} allows an explicit assessment in the transform domain. Defining $\bar a_\ell(x)\equiv \mathbb{E}A_1^\ell(x)/\ell!$ and $\bar \theta(\ell)\equiv \theta(\ell)/\ell!$, multiplying the entire equation \eqref{eq: recusrion ODE} by $\beta e^{-\beta x}$ and integrating over the positive halfline yields (by swapping the order of the two integrals)
    \begin{align}\label{eq103}
        a_\ell(\beta)&\equiv \beta \int_0^\infty e^{-\beta x} \bar a_\ell(x)\,{\rm d}x = \beta \int_0^\infty e^{-\beta x}\sum_{i=0}^{\ell-1}\bar\theta(\ell-i)\,\int_0^x \bar a_i(y)\,{\rm d}y\,{\rm d}x= \sum_{i=0}^{\ell-1}\bar\theta(\ell-i)\,a_i(\beta).
    \end{align}
    Now multiplying \eqref{eq103} by $z^\ell$ (with $|z|<1$) and summing over $\ell\in{\mathbb N}$, denoting $\Theta(z)\equiv \sum_{\ell=1}^\infty z^\ell \bar\theta(\ell)$, 
    \begin{align}
        {\mathscr A}(z,\beta)&\equiv \sum_{\ell =1}^\infty z^\ell a_\ell(\beta) =
        \sum_{\ell =1}^\infty z^\ell\sum_{i=0}^{\ell-1}\bar\theta(\ell-i)\,a_i(\beta)\\ \notag
        \notag &= \sum_{i=0}^\infty \left(\sum_{\ell=i+1}^\infty z^{\ell-i} \,\bar\theta(\ell-i)\right) \,z^i a_i(\beta)= \Theta(z) \,\left( {\mathscr A}(z,\beta) + a_0(\beta)\right).
    \end{align}
    We conclude that the double transform ${\mathscr A}(z,\beta)$ satisfies
    \begin{equation}
    {\mathscr A}(z,\beta) = \frac{\Theta(z)}{1-\Theta(z)} a_0(\beta)
    \end{equation}
    for $|z|<1.$ Here $a_0(\beta)=\beta\int_0^\infty e^{-\beta x}\,{\mathbb E}\tau_x\,{\rm d}x=-1/(\beta\varrho)$.
\end{remark}

 \section{The area beneath $\overline{W}_V(\cdot)$ until an exponential time}\label{sec: transform}
    In this paper, we devised an algorithm to evaluate (joint) high-order moments of $A_f$, $D_g$, and $\tau$, in the general L\'evy-driven setup that we discussed. In this section we provide a potential application of this result. Namely, consider the case when: (1) $J(\cdot)$ is a CPP with an arrival rate $\lambda$ and nonnegative jumps having cdf $G(\cdot)$ minus a unit drift. (2) $V$ is distributed according to $G(\cdot)$. Recall the definition of the Skorokhod-reflected process $\overline{W}_V(\cdot)\equiv \overline{W}(\cdot)$ from \eqref{SKOR}. We now show that the main results of  this paper can be applied in order to identify the LST of
    \begin{equation}\label{AVT}
    \overline{\color{black}{\boldsymbol A}\color{black}}_V(T_\beta)\equiv\int_0^{T_\beta}\overline{W}_V(t)\,{\rm d}t\,=
    \int_0^{T_\beta}\int_0^\infty \overline{W}_x(t)\,{\rm d}G(x)\,{\rm d}t,
    \end{equation}
    where $T_\beta$ is an exponentially distributed random variable with rate $\beta$ which is independent from both $V$ and $J(\cdot)$. 
By first dividing this quantity by $\beta$, and then (numerically) inverting it   with respect to $\beta$, we obtain access to the distribution of $\overline {\color{black}{\boldsymbol A}\color{black}}_V(t)$.
Algorithms for numerical inversion can be found in e.g.\  \cite{Abate1995}.

The line of reasoning is as follows.
\begin{itemize}
    \item[$\circ$]
     Our goal is to identify the LST of \eqref{AVT}, i.e., 
\begin{equation}\Omega(\alpha,\beta)\equiv  {\mathbb E} \, e^{-\alpha \overline {\color{black}{\boldsymbol A}\color{black}}_V(T_\beta)}=\int_0^\infty {\mathbb E} \, e^{-\alpha \overline {\color{black}{\boldsymbol A}\color{black}}_x(T_\beta)}\,{\rm d}G(x). \end{equation}
\item[$\circ$]
For every $x,t\geqslant 0$ define \color{black} ${\boldsymbol A}_x(t)\equiv\int_0^t W_x(s)\,{\rm d}s$ \color{black} and \color{black} $\overline {{\boldsymbol A}}_x(t)\equiv\int_0^t \overline W_x(s)\,{\rm d}s$\color{black}. 
 In particular, observe that 
\begin{equation}\Xi(\alpha,\beta)\equiv \int_0^\infty {\mathbb E} \, e^{-\alpha  \overline {\color{black}{\boldsymbol A}\color{black}}_x(\tau_x)}\textbf{1}_{\{\tau_x<T_\beta\}} {\rm d}G(x)= {\mathbb E}\,e^{-\alpha A-\beta \tau}\,,
\end{equation}
where $A\equiv \int_0^\tau W_V(t)\,{\rm d}t.$
Observe that we (in principle) have access to $\Xi(\alpha,\beta)$, as, by applying the results of the previous sections, we have an algorithm to compute the joint moments of $A$ and $\tau$ in terms of the model primitives; note that
\begin{equation}\Xi(\alpha,\beta)= \sum_{m=0}^\infty\sum_{n=0}^\infty \frac{\alpha^m\beta^n}{m!n!}{\mathbb E}\,A^m\tau^n.\end{equation}
In practical terms, if one were able to numerically evaluate $\Xi(\alpha,\beta)$, a pragmatic approach is to approximate it by a truncated Taylor series, e.g.,
 \begin{equation}\Xi(\alpha,\beta) \approx \Xi_M(\alpha,\beta)\equiv \sum_{m,n\in\mathbb{N}_0:m+n\leqslant M} \frac{\alpha^m\beta^n}{m!n!}{\mathbb E}\,A^m\tau^n\end{equation}
 for some $M\in{\mathbb N}$ chosen such that a given accuracy level is achieved.
 \item[$\circ$]
 In Theorem \ref{thm: transform} we succeed in expressing $\Omega(\alpha,\beta)$ in terms of $\Xi(\alpha,\beta)$. We have thus developed a way to evaluate the transform $\Omega(\alpha,\beta)$.
\end{itemize}
 In our analysis, we also use the transform
\begin{equation}\label{eq: pi}
      \Pi(\alpha,\beta)\equiv \int_0^\infty \frac{\beta}{\alpha x+\beta-\bar\varphi(\alpha)}\,{\rm d}G(x)=\int_0^\infty \beta \,\widetilde G(\alpha s)\,e^{-(\beta-\bar\varphi(\alpha))s}\,{\rm d}s,
  \end{equation}
   where $\bar\varphi(\alpha)\equiv \int_0^1\varphi(\alpha t)\,{\rm d}t$ and $\widetilde{G}(\cdot)$ is the LST of $G(\cdot)$. Especially, notice that the equality in \eqref{eq: pi} follows by first writing (in the rightmost expression, that is) $\widetilde G(\alpha s)$ as an integral and then swapping the order of the two integrals.

 \begin{theorem}\label{thm: transform}
For any $\alpha,\beta>0$, 
\begin{equation}\label{eq: exponential}
    \Omega\left(\alpha,\beta\right)=\frac{\displaystyle (\lambda+\beta)\Pi\left(\alpha,\beta\right)+\beta\,\Xi\left(\alpha,\beta\right)\left(1-\frac{\lambda+\beta}{\beta-\overline{\varphi}(\alpha)}\right)}{\displaystyle \lambda+\beta-{\lambda\,\Xi\left(\alpha,\beta\right)}}.
\end{equation}
 \end{theorem}

 \begin{proof}
 First, for every $x\geqslant 0$, write
 \begin{equation}\label{eq: decomposition}
 {\mathbb E} \, e^{-\alpha \overline {\color{black}{\boldsymbol A}\color{black}}_x(T_\beta)}={\mathbb E} \, e^{-\alpha \overline {\color{black}{\boldsymbol A}\color{black}}_x(T_\beta)} \textbf{1}_{\{\tau_x<T_\beta\}}+{\mathbb E} \, e^{-\alpha \overline {\color{black}{\boldsymbol A}\color{black}}_x(T_\beta)} \textbf{1}_{\{\tau_x\geqslant T_\beta\}}.\end{equation}
 Owing to the underlying Markovian structure, and distinguishing between the scenario that a second busy period starts before the killing epoch $T_\beta$ and the scenario that it does not, we can rewrite the first term on the right-hand side to
 \begin{align}
    {\mathbb E} \, e^{-\alpha \overline {\color{black}{\boldsymbol A}\color{black}}_x(T_\beta)} \textbf{1}_{\{\tau_x<T_\beta\}}&\equiv  {\mathbb E} \, e^{-\alpha \overline {\color{black}{\boldsymbol A}\color{black}}_x(\tau_x)} \textbf{1}_{\{\tau_x<T_\beta\}}
    \left(\frac{\lambda}{\lambda+\beta}\Omega(\alpha,\beta)+\frac{\beta}{\lambda+\beta}\right).
 \end{align} The second term on the right-hand side of \eqref{eq: decomposition} equals
 \begin{align}
     {\mathbb E} \, e^{-\alpha \overline {\color{black}{\boldsymbol A}\color{black}}_x(T_\beta)} \textbf{1}_{\{\tau_x\geqslant T_\beta\}}&\equiv {\mathbb E} \, e^{-\alpha {\color{black}{\boldsymbol A}\color{black}}_x(T_\beta)} \textbf{1}_{\{\tau_x\geqslant T_\beta\}}=
{\mathbb E} \, e^{-\alpha  \color{black}{\boldsymbol A}\color{black}_x(T_\beta)} - {\mathbb E} \, e^{-\alpha \color{black}{\boldsymbol A}\color{black}_x(T_\beta)}\textbf{1}_{\{\tau_x<T_\beta\}}.\notag 
 \end{align}
Then,
 \begin{align}
     {\mathbb E} \, e^{-\alpha \color{black}{\boldsymbol A}\color{black}_x(T_\beta)} &= \int_0^\infty \beta e^{-\beta s} \,{\mathbb E}\exp\left(-\alpha\int_0^s \left[x+J(t)\right]\,{\rm d}t\right){\rm d}s \\
     &= \int_0^\infty \beta e^{-(\alpha x+\beta) s} \,e^{\bar\varphi(\alpha)\,s}\,{\rm d}s= \frac{\beta}{\alpha x+\beta-\bar\varphi(\alpha)},\notag
 \end{align}
 where in the penultimate equality \cite[Equation (4)]{Blanchet2013} has been used. Also, using the memoryless property of the exponential distribution, 
 \begin{align}
      {\mathbb E} \, e^{-\alpha \color{black}{\boldsymbol A}\color{black}_x(T_\beta)}\textbf{1}_{\{\tau_x<T_\beta\}}&=
       {\mathbb E} \, e^{-\alpha  \overline {\color{black}{\boldsymbol A}\color{black}}_x(\tau_x)}\textbf{1}_{\{\tau_x<T_\beta\}} \,
        {\mathbb E} \, e^{-\alpha A_0(T_\beta)}= {\mathbb E} \, e^{-\alpha  \overline {\color{black}{\boldsymbol A}\color{black}}_x(\tau_x)}\textbf{1}_{\{\tau_x<T_\beta\}} \, \frac{\beta}{\beta-\bar\varphi(\alpha)}.\notag
 \end{align}
 Upon combining the above, we can express $\Omega(\alpha,\beta)$ in itself. Indeed, by integrating the full expression with respect to ${\rm d}G(\cdot)$, 
 \begin{equation}
   \Omega(\alpha,\beta) = \Xi(\alpha,\beta)   \left(\frac{\lambda}{\lambda+\beta}\Omega(\alpha,\beta)+\frac{\beta}{\lambda+\beta}\right) + \Pi(\alpha,\beta) - \Xi(\alpha,\beta) \, \frac{\beta}{\beta-\bar\varphi(\alpha)}.
 \end{equation}
 This finishes the proof. \end{proof}
 
\section{Demonstration}\label{sec: demonstration}
To demonstrate our method producing expression for moments, we show that it provides us with an easy derivation of the existing results \eqref{eq: Iglehart1971} and \eqref{eq: Cohen1978}. Consider the special case when \color{black} $\sigma=0$ and \color{black} $J(\cdot)$ is a CPP with arrival rate $\lambda$ and nonnegative jumps with the distribution function $G(\cdot)$, minus a drift of rate one. In addition, the initial push $V$ is distributed according to the distribution function $G(\cdot)$ as well, and it is assumed that $\varrho=\rho-1=\lambda\mu_1-1<0$ to ensure stability. 

It is easy to verify that $-\varphi_1=-\varrho=1-\rho$ and $(-1)^k\varphi_k=\lambda\mu_k$ for any $k\in\{2,3,\ldots\}$, with (as before) $\mu_i$ denoting the $i$-th moment of the jump distribution of $J(\cdot)$. Consequently,  Theorem \ref{thm: Takacs} (with \eqref{eq: Takacs expectation}) implies that, for every $k\in{\mathbb N}$,
\begin{equation}\label{eq: Takacs M/G/1}
    (-1)^k\psi_k=-\frac{\lambda}{(k+1)\varphi_1}\sum_{i=0}^{k-1}\binom{k+1}{i}(-1)^i\psi_i\mu_{k+1-i};
\end{equation}
it is instructive to compare \eqref{eq: Takacs M/G/1} with \cite[Theorem 5]{Takacs1962}. In particular, we obtain that
\begin{equation}\label{eq: psi1}
    -\psi_1=-\frac{\varphi_2}{2\varphi_1}=\frac{\lambda\mu_2}{2(1-\rho)}\,.
\end{equation}
Therefore, an insertion of this expression into Corollary \ref{cor: polynomial} (with $k=1$) implies that 
\begin{equation}
A_{\mathcal{P}_1}=-\frac{1}{\varrho}\left(\frac{\mu_2}{2}-\mu_1\psi_1\right)=\frac{\mu_2}{2(1-\rho)^2}
\end{equation}
which is in line with \eqref{eq: Iglehart1971}. 

The next goal is to compute the second moment of $A_{\mathcal{P}_1}$. To this end, apply \eqref{eq: Takacs M/G/1} in order to compute $(-1)^2\psi_2$ and $(-1)^3\psi_3$. Specifically, 
\begin{align}\label{eq:psi2}
    (-1)^2\psi_2&=\frac{\lambda}{3(1-\rho)}\sum_{i=0}^1\binom{3}{i}(-1)^i\psi_i\mu_{3-i}\\&\nonumber=\frac{\lambda}{3(1-\rho)}\left[\mu_3+3\mu_2\cdot\frac{\lambda\mu_2}{2(1-\rho)}\right]=\frac{\lambda\mu_3}{3(1-\rho)}+\frac{\lambda^2\mu_2^2}{2(1-\rho)^2}\,,
\end{align}
and
\begin{align}\label{eq:psi3}
    (-1)^3\psi_3&=\frac{\lambda}{4(1-\rho)}\sum_{i=0}^2\binom{4}{i}(-1)^i\psi_i\mu_{4-i}\\&\nonumber=\frac{\lambda}{4(1-\rho)}\left[\mu_4+\frac{2\lambda\mu_2\mu_3}{1-\rho}+\frac{2\lambda\mu_3\mu_2}{1-\rho}+\frac{3\lambda^2\mu_2^3}{(1-\rho)^2}\right]\,.
\end{align}
Now, an application of Theorem \ref{thm: Gamma} and an integration of both sides  of  \eqref{eq: polynomial22} (when $k_1=k_2=1$) with respect to ${\rm d}G(x)$ imply that 
\begin{equation}\label{eq: second moment}
    \mathbb{E}A^2_{\mathcal{P}_1}=\frac{2}{(1-\rho)^2}\sum_{i=0}^{1}\binom{1}{i}\frac{(-1)^{i}\psi_i}{2-i}\sum_{j=0}^{3-i}\binom{3-i}{j}\frac{(-1)^j\psi_j}{4-i-j}\mu_{4-i-j}\,.
\end{equation}
Thus, an insertion of \eqref{eq: psi1}, \eqref{eq:psi2} and \eqref{eq:psi3} into \eqref{eq: second moment} yields the same result as in \eqref{eq: Cohen1978}: 
\begin{align}
    (1-\rho)^2\mathbb{E}A^2_{\mathcal{P}_1}&=-2\psi_1\sum_{i=0}^2\binom{2}{i}\frac{(-1)^i\psi_i}{3-i}\mu_{3-i}+2\cdot\frac{1}{2}\sum_{i=0}^3\binom{3}{i}\frac{(-1)^i\psi_i}{4-i}\mu_{4-i}\\&\nonumber=\frac{\lambda\mu_2}{1-\rho}\left[\frac{\mu_3}{3}+\frac{\lambda\mu_2^2}{2(1-\rho)}+\frac{\rho\mu_3}{3(1-\rho)}+\frac{\lambda\rho\mu_2^2}{2(1-\rho)^2}\right]\\&\nonumber\:\:\:\:\:+\frac{\mu_4}{4}+\frac{\lambda\mu_2\mu_3}{2(1-\rho)}+\frac{\lambda\mu_3\mu_2}{2(1-\rho)}+\frac{3\lambda^2\mu_2^3}{4(1-\rho)^2}\\&\nonumber\:\:\:\:\:+\frac{\rho}{4(1-\rho)}\left[\mu_4+\frac{2\lambda\mu_2\mu_3}{1-\rho}+\frac{2\lambda\mu_3\mu_2}{1-\rho}+\frac{3\lambda^2\mu_2^3}{(1-\rho)^2}\right]\\&=\frac{\mu_4}{4(1-\rho)}+\frac{4\lambda\mu_2\mu_3}{3(1-\rho)^2}+\frac{5\lambda^2\mu_2^3}{4(1-\rho)^3}\,.\nonumber
\end{align}
We can recover known results, but also derive new ones. Motivated by the regenerative method for L\'evy processes with a secondary jump input (as discussed in Section \ref{sec:introduction}), we compute the joint moment of $A$ and $\tau$ in a similar fashion. Namely, observe that
\begin{align}
(1-\rho)^2\mathbb{E}A_{\mathcal{P}_1}A_{\mathcal{P}_0}&=\sum_{i_2=0}^{2}\binom{2}{i_2}\frac{(-1)^{i_2}\psi_{i_2}}{3-i_2}\mu_{3-i_2}+\sum_{i_1=0}^{1}\frac{(-1)^{i_1}\psi_{i_1}}{2-i_1}\sum_{i_2=0}^{2-i_1}\binom{2-i_1}{i_2}\frac{(-1)^{i_2}\psi_{i_2}}{3-i_1-i_2}\mu_{3-i_1-i_2}\nonumber\\&=\frac{\mu_3}{3}+\frac{\lambda\mu_2^2}{2(1-\rho)}+\frac{\rho\mu_3}{3(1-\rho)}+\frac{\rho\lambda\mu_2^2}{2(1-\rho)^2}\:+\nonumber\\&\hspace{1.4cm}\frac{1}{2}\sum_{i_2=0}^{2}\binom{2}{i_2}\frac{(-1)^{i_2}\psi_{i_2}}{3-i_2}\mu_{3-i_2}+\frac{\lambda\mu_2}{2(1-\rho)}\sum_{i_2=0}^{1}\frac{(-1)^{i_2}\psi_{i_2}}{2-i_2}\mu_{2-i_2}\nonumber\\&\nonumber=\frac{3}{2}\left[\frac{\mu_3}{3(1-\rho)}+\frac{\lambda\mu_2^2}{2(1-\rho)^2}\right]+\frac{\lambda\mu_2^2}{4(1-\rho)^2}=\frac{\mu_3}{2(1-\rho)}+\frac{\lambda\mu_2^2}{(1-\rho)^2}\,,\nonumber
\end{align}
and hence 
\begin{equation}\label{eq: joint a tau}
    \mathbb{E}A\tau =\mathbb{E}A_{\mathcal{P}_1}A_{\mathcal{P}_0}=\frac{\mu_3}{2(1-\rho)^3}+\frac{\lambda\mu_2^2}{(1-\rho)^4}\,.
\end{equation}
To the best of our knowledge, the identity \eqref{eq: joint a tau} is new. The above computation demonstrates that it is straightforward to obtain expressions of moments. We in addition note that the underlying recursive structure lends itself well for being evaluated by means of symbolic software. 

\section{Discussion and directions for follow-up research}\label{sec: open problems}
In this work we solved the problem of identifying the high-order moments of the vector \eqref{eq: the vector}. Specifically, we developed a recursive procedure that easily produces these moments in terms of the model primitives. Our findings in particular provide a solution for a long-standing problem posed by Iglehart \cite{Iglehart1971} and Cohen \cite{Cohen1978} about existence of a constructive algorithm yielding the moments of $A_f$ for the case that $f(\cdot)={\rm Id}$. As it was pointed out in Section \ref{sec:introduction}, being able to  solve this open problem is relevant in the context of various  applied probability models, e.g., (1) the computation of the conditional moments of the `area in red' in a L\'evy-driven model given the undershoot, and (2)~the estimation of functionals of the limiting distribution of a L\'evy process with secondary jump input.     

We believe that our results, and the approach followed, naturally give rise to several follow-up questions. A few particularly interesting ones are:

\begin{itemize}
    \item[$\circ$] The current results may be applied to approximate the LST of $A_k(x)$. We expect this approximation to be useful when assessing the tail probability $\mathbb{P}\left\{A_k(x)>u\right\}$ as $u\to\infty$. For some existing works in this direction, see e.g.\ \cite{Arendarczyk2014,Blanchet2013,Borovkov2003}. In case $x$ is sampled according to the distribution function $G_V(\cdot)$, a large deviation of the area can be due to a non-trivial interplay between the secondary jump attaining an unusually high value and the driving L\'evy process $J(\cdot)$ being after time $0$ systematically higher than usual. 
    
    \item[$\circ$] More generally, let $f(\cdot)$ be a nonnegative continuous function.  The Weierstrass approximation theorem states that $f(\cdot)$ can be approximated arbitrarily closely by polynomials (e.g., Bernstein polynomials). This suggests that our approach could be used in order to approximate the moments of $A_f$ for any $f(\cdot)$ (under some regularity conditions).  

    \item[$\circ$] Another natural question concerns the identification of conditions under which the sequence 
    \begin{equation}
        \frac{A_k(x)-\mathbb{E}A_k(x)}{\sqrt{\text{Var}\,A_k(x)}}\ \ , \ \ x>0
    \end{equation}
    converges as $x\to\infty$ (for given $k$) to some limiting distribution.  

    \item[$\circ$] Consider a sequence of cost processes \[\left\{A_k^{(n)}(\cdot):n\in\mathbb{N}\right\}.\] If  $\varphi_n(\cdot)$ is the L\'evy exponent corresponding to the $n$-th model, what should be assumed on the sequence $\varphi_n(\cdot)$ in order to ensure,  in a proper space of functions, the weak convergence $A_k^{(n)}(\cdot)$ to some limiting process $A(\cdot)$? For a similar analysis of a related process, we refer to   \cite[Section 7]{Jacobovic2023} and  \cite[Section 6]{Jacobovic2022}.

\item[$\circ$] Note that in the current work, we  only consider the case that the driving L\'evy process is a superposition of a drifted Brownian motion and a subordinator, which is a special case of the broader class of  spectrally positive \color{black}  L\'evy processes.  When trying to extend our results to the  spectrally positive \color{black}  L\'evy case, a major bottleneck concerns the generalization of the Takacs' recursion of Theorem \ref{thm: Takacs}; recall that in our proofs we intensively used that our driving L\'evy process was the aggregate of a  drifted Brownian motion and a subordinator.

\item[$\circ$] In the context of Section \ref{sec: transform}, one may be curious to know whether the relation \eqref{eq: exponential} which appears in Theorem \ref{thm: transform} can be generalized by considering classes of processes that are richer than just CPP. Furthermore, Theorem \ref{thm: transform} is valid only if $f(\cdot)={\rm Id}$, but one could pursue analogous results for more general classes of functionals.  
\end{itemize}
\medskip

{\it Acknowledgment}. We would like to thank Onno Boxma for his helpful comments on an earlier draft of the current manuscript. We are also grateful to H\'el\`ene Gu\'erin and Jean-Fran\c{c}ois Renaud for pointing out the relevance of our results in the insurance context. \color{black} We are indebted to the anonymous referee for providing useful comments which helped us to significantly improve the presentation of our results. \color{black}

\bibliographystyle{plain}

\appendix
\section{Proof of Lemma \ref{lemma: psi finite}}
Fix $m,k\in\mathbb{N}$. Note that $c+\mathbb{E}\overline{\Psi}_m(1)<0$ and, therefore, it is possible to decompose $c$ into $c_1+c_2$ such that $c_1<0$ is negative and $c_2$ satisfies $c_2< c_2+\mathbb{E}\overline{\Psi}_m(1) \leqslant c_2+\mathbb{E}\Psi_{0:\infty}(1)<0$. Define two stochastic processes:
    \begin{enumerate}
        \item $P(t)\equiv \sigma B(t)+c_1t$ for every $t\geqslant0$.

        \item $Q(t)\equiv Q_m(t)\equiv V\wedge m+\overline{\Psi}_m(t)+c_2t$ for every $t\geqslant0$;  \end{enumerate}
         i.e., $P(\cdot)$ is a drifted Brownian motion starting at zero and $Q(\cdot)$ is a CPP with drift starting at $V\wedge m$. 
          Denote $y^+\equiv \max\{y,0\}$ for $y\in\mathbb{R}$ and define two random times:
         \begin{itemize}
             \item[$\circ$] The last exit time of $Q(\cdot)$ from the set $(0,\infty)$, i.e.,
             \begin{equation}
             \upsilon_Q\equiv\sup\left\{t\geqslant 0: Q(t)>0\right\}
         \end{equation}
         
         \item [$\circ$] The first time after $\upsilon_Q$ that $P^+(\cdot)$ hits the set $\{0\}$, i.e.,  
         \begin{equation}
         \upsilon_{P,Q}\equiv\inf\left\{t\geqslant \upsilon_Q: P^+(t)=0\right\}\,.   \end{equation}
         
         \end{itemize}
         As $\mathbb{E}\overline{\Psi}_m(1)+c_2\color{black}<0$ and $\mathbb{E}\,P(1)<0$, it holds that $Q(t)\xrightarrow{t\to\infty}-\infty$ and $P(t)\xrightarrow{t\to\infty}-\infty$, $\mathbb{P}$-a.s.\ and hence $\upsilon_Q<\infty$ and $\upsilon_{P,Q}<\infty$, $\mathbb{P}$-a.s.
         Observe also that at each time  $0\leqslant t  <\widetilde{\tau}(m)$, we have $P(t)>0$ or $Q(t)>0$ (or both). Thus, since $P(\upsilon_{P,Q})=0$ and $Q(\upsilon_{P,Q})\leqslant 0$, we have that $\widetilde{\tau}(m)\leqslant \upsilon_{P,Q}$.
         
        Notice that for every $0\leqslant t\leqslant\widetilde{\tau}(m)$ we have that $0\leqslant \widetilde{W}_m(t)=P(t)+Q(t)$ and hence  
        \begin{equation}
            0\leqslant \widetilde{W}^k_m(t)\leqslant \left[P^+(t)+Q^+(t)\right]^k\ \ , \ \ \forall t\in[0,\widetilde{\tau}(m)]\,.
        \end{equation}
         Thus, by applying the Minkowski inequality, it follows that
        \begin{align} \notag \left[\mathbb{E}\int_0^{\widetilde{\tau}(m)}\widetilde{W}_m^k(t){\rm d}t\right]^{\frac{1}{k}}&\leqslant \left\{\mathbb{E}\int_0^{\upsilon_{P,Q}}\left[P^+(t)+Q^+(t)\right]^k{\rm d}t\right\}^{\frac{1}{k}}\\&   \leqslant \left\{\mathbb{E}\int_0^{\upsilon_{P,Q}}\left[P^+(t)\right]^k{\rm d}t\right\}^{\frac{1}{k}}+\left\{\mathbb{E}\int_0^{\upsilon_Q}\left[Q^+(t)\right]^k{\rm d}t\right\}^{\frac{1}{k}}\,. \label{eq: upper bound CG}
        \end{align}
        The goal of the remainder of the proof is to point out why both terms in the upper bound \eqref{eq: upper bound CG} are finite. 
        We start by studying the second term. Recall that $\mathbb{E}\,Q(1)<0$ and hence $Q(t)\xrightarrow{t\to\infty}-\infty$.  Let $\tau_Q$ be the first time that $Q(\cdot)$ hits the level $0$ and also define
        \begin{equation}
            I_1(Q)\equiv\int_0^{\tau_Q}Q^k(t){\rm d}t=\int_0^{\tau_Q}\left[Q^+(t)\right]^k{\rm d}t.
        \end{equation}
        Thus, \color{black} by applying the strong Markov property of $Q(\cdot)$ (see  \cite[Theorem 3.1]{Kyprianou2006}), we have that
        \begin{equation}
         \int_0^{\upsilon_Q}\left[Q^+(t)\right]^k{\rm d}t-I_1(Q)=\color{black} \int_0^\infty\left[Q^+(t)\right]^k{\rm d}t-\int_0^{\tau_Q}\left[Q^+(t)\right]^k{\rm d}t   
        \end{equation}
        has a compound geometric distribution, i.e., is distributed as a geometric number of independent and identically distributed increments (with the sizes of these increments being independent of the geometric number of terms).  This compound geometric distribution is characterized as follows:
        \begin{itemize}
            \item[$\circ$] The probability of success is equal to the probability that $Q(\cdot)$ will not hit $(0,\infty)$ in the future, given that its current state is zero. Keeping in mind that  $Q(t)\xrightarrow{t\to\infty}-\infty$, $\mathbb{P}$-a.s., and recalling that $Q(\cdot)$ is a CPP, this probability is positive and less than one. 

            \item[$\circ$] To represent the distribution that defines the increments, let $\tau_Q'$ be the first time after $\tau_Q$ that $Q(\cdot)$ hits $(0,\infty)$. In addition, let $\tau''_Q$ be the first time after $\tau_Q'$ in which $Q(\cdot)$ hits zero. Note that $\tau'_Q$ (and hence also $\tau''_Q$) may be infinity with a probability strictly between zero and one. Then, the underlying distribution of the increments is the conditional distribution of the integral
            \begin{equation}
            I_2(Q)\equiv\int_{\tau'_Q}^{\tau''_Q}\left[Q^+(t)\right]^k{\rm d}t\,,
            \end{equation}
            given the event $\{\tau'_Q<\infty\}$.  
        \end{itemize} 

        To construct the stochastic upper bound on the quantity $I_2(Q)$, recall that the jumps of $Q(\cdot)$ are bounded by $m$. In addition, denote the first time that $t\mapsto m+Q(t)$ hits zero by $\widetilde{\tau}_Q$, and recall that $Q(\cdot)$ is a strong Markov process. Thus, if $\widetilde{V}$ has the distribution of the overshoot $Q(\tau'_Q)$ given the event $\{\tau'_Q<\infty\}$ and it is independent from $Q(\cdot)$, then $I_2(Q)$ is distributed like $I_2'(Q)$ which is given by 
        \begin{equation}
        I_2'(Q)\equiv \int_0^{\upsilon'_Q}\left[\widetilde{V}-V+Q(t)\right]^k{\rm d}t\leqslant \upsilon'_Q\,m^k (N_Q+1)^k   
        \end{equation}
        where $\upsilon'_Q$ is the first time that $t\mapsto \widetilde{V}-V+Q(t)$ hits zero and $N_Q$ is the number of jumps that $Q(\cdot)$ has during the time interval $(0,\upsilon'_Q)$.  Thus, by the discussion provided in the introduction regarding \eqref{eq: lst equation} and \eqref{eq: LST CPP!}, we observe that all the joint moments of $(\upsilon'_Q,N_Q)$ are finite, i.e., we find that 
        \begin{equation}
    \mathbb{E}\left[I_2(Q)\,|\,\tau_Q'<\infty\right]=\frac{\mathbb{E}I_2(Q)\textbf{1}_{\{\tau'_Q<\infty\}}}{\mathbb{P}\{\tau'_Q<\infty\}}\leqslant\frac{\mathbb{E}I'_2(Q)}{\mathbb{P}\{\tau'_Q<\infty\}}<\infty\,,
        \end{equation}
        as required. The proof that $\mathbb{E}I_1(Q)<\infty$ follows via the same arguments which were used to justify that $\mathbb{E}I_2'(Q)<\infty$.
    
   It is left to show that the first term in \eqref{eq: upper bound CG} is finite. 
    To this end, as $c_1<0$, we have that $P(t)\leqslant \sigma B(t)$ for any $t\geqslant 0$, which implies that $P^+(t)\leqslant \sigma B^+(t)\leqslant \sigma |B(t)|$ for any $t\geqslant 0$. Noting that $\upsilon_Q$ and the process $B(\cdot)$ are independent, we conclude that, conditional on $\upsilon_Q$, the random time $\upsilon_{P,Q}$ is a stopping time determined by $B(\cdot)$.  Therefore, \cite[Theorem 2.1]{Peshkir1998} implies that there is a constant number $\iota\in(0,\infty)$ such that
    \begin{align}\label{eq: bound CG2}
       \mathbb{E}\left[\int_0^{\upsilon_{P,Q}}
       \left[P^+(t)\right]^k{\rm d}t\right] &
       \leqslant \sigma^k\,\mathbb{E}\left[\mathbb{E}\int_0^{\upsilon_{P,Q}}\left|B(t)\right|^k{\rm d}t\,\bigg|\,\upsilon_Q\right]\\
       &\nonumber\leqslant \sigma^k\,\mathbb{E}\left[\iota\cdot \mathbb{E}\upsilon_{P,Q}^{1+\frac{k}{2}}\,\bigg|\,\upsilon_Q\right]=\iota\cdot \sigma^k\cdot \mathbb{E}\upsilon_{P,Q}^{1+\frac{k}{2}}.
    \end{align}
    The first step in showing that $\upsilon_{P,Q}$ has finite moments is to show that $\upsilon_Q$ has finite moments. Note that $\upsilon_Q=\tau_Q+(\upsilon_Q-\tau_Q)$. As explained in the introduction (see, the discussion about \eqref{eq: LST CPP!}), $\tau_Q$ has finite moments. In addition, due to the strong Markov property of $t\mapsto\color{black}\overline{\Psi}_m(t\color{black})+c_2t$, conditionally on $\tau_Q$, the difference $\upsilon_Q-\tau_Q$ has a compound geometric distribution, that is characterized as follows:
    \begin{itemize}
        \item[$\circ$] The probability of success which is equal to the probability that $Q(\cdot)$ will not \color{black} hit zero in the future given the initial state is zero. Note that as the driving process is of CPP type, this probability is strictly between zero and one. 

        \item[$\circ$] The distribution of the increments corresponds to a convolution of two distributions:\newline 
        \begin{enumerate}
            \item[1.] The distribution of the time to ruin (given that ruin happens) in a Cram\'er-Lundberg model with: \textit{(i)} Initial capital $0$. \textit{(ii)} Premium rate $-c_2$ (where it is kept in mind that $c_2<0$).
            \textit{(iii)} The claim distribution is the distribution of $V\wedge m$. \textit{(iv)} The claim arrival rate is equal to the arrival rate of the jumps in $\overline{\Psi}_m$. Importantly,  the arrival rate is finite and the claim sizes are all bounded by $m$. Therefore, the main result of \cite{Delbaen1990} implies that all moments of this distribution are finite.\newline

            \item[2.] The distribution of the first busy period in a FCFS M/G/1 queue with an initial workload which is distributed like the overshoot of $Q(\cdot)$ at the first moment that it hits $(0,\infty)$ after $\tau_Q$. Importantly, as the jumps of $Q(\cdot)$ are all bounded by $m$ (and so is the overshoot), by the discussion provided in relation to  \eqref{eq: LST CPP!}, we conclude that all the moments of this distributions are finite. 
        \end{enumerate}
    \end{itemize}
    Since the probability of success lies in $(0,1)$ and all the moments of the increments are finite, it follows that the moments of $\upsilon_Q$ are all finite as well.  Thus, it is enough to prove that $\widetilde{\upsilon}\equiv\upsilon_{P,Q}-\upsilon_Q$ has finite moments. To this end, recall that $\upsilon_{P,Q}$ is the first time after $\upsilon_Q$ that $P^{+}(\cdot)$ hits the origin, which implies that 
    \begin{equation}
        \widetilde{\upsilon}\leqslant\inf\left\{t\geqslant0: \sigma |B\left(\upsilon_Q\right)|+\sigma \left[B(t+\upsilon_Q)-B(\upsilon_Q)\right]+c_1t=0\right\}\equiv\widetilde{\upsilon}^\circ
    \end{equation}
    Therefore, by conditioning on $\upsilon_Q$, which is independent of $B(\cdot)$, we obtain that, for any $s\in {\mathbb N}_0$, 
\begin{equation}\mathbb{E}\,[\widetilde{\upsilon}^s]\leqslant\mathbb{E}\left[(\widetilde{\upsilon}^\circ)^s\right]=\mathbb{E}\Big[\mathbb{E}\left[(\widetilde{\upsilon}^\circ)^s\,\big|\,\upsilon_Q,\left\{B(t):0\leqslant t\leqslant \upsilon_Q\right\}\right]\Big]. 
    \end{equation}
    In particular, the strong Markov property of $B(\cdot)$ implies that the inner (conditional) expectation is equal to the $s$-th moment of 
    \begin{equation*}
        \inf\left\{t\geqslant0:U+\sigma B(t)+c_1t=0\right\}
    \end{equation*}
    where $U$ is a random variable having the distribution of $\sigma\,|B(\upsilon_Q)|$, being independent of $B(\cdot)$. Thus, \cite[Proposition 2.5]{Abundo2017} implies that, for any positive integer $s$, $\mathbb{E}[(\widetilde{\upsilon}^\circ)^s]$ is a linear combination of the first $s$ moments of $|B(\upsilon_Q)|$. Note that $\upsilon_Q$ and $B(\cdot)$ are independent, i.e., by proper conditioning and un-conditioning for any $w=1,2,\ldots,s$ we have $\mathbb{E}|B(\upsilon_Q)|^w=\mathbb{E}\upsilon_Q^{w/2}\,\mathbb{E}|Z|^w$ where $Z$ is a standard univariate Gaussian random variable. Hence,  the required property follows from the fact that both $|Z|$ and $\upsilon_Q$ have finite moments. \hfill $\Box$

\end{document}